\documentclass{article}

\usepackage{graphicx} 
\usepackage{algorithm,algorithmic}
\usepackage{amsmath}
\usepackage{amsthm}
\usepackage{amssymb}
\usepackage{color}
\usepackage{bm}
\usepackage{float}
\usepackage{hyperref}
\newtheorem{lemma}{Lemma}[section]
\newtheorem{theorem}{Theorem}[section]
\newtheorem{remark}{Remark}[section]

\usepackage[textsize=tiny]{todonotes}
\usepackage{fullpage}

 \usepackage{cite}

\title{On a blended Anderson / NGMRES acceleration for the Navier-Stokes Picard iteration: convergence analysis and adaptive depth selection
}
\author{Yunhui He\thanks{Department of Mathematics, University of Houston, 3551 Cullen Blvd, Room 641, Houston, Texas 77204-3008, USA 
  (\tt{yhe43@central.uh.edu}).}
\and Leo G Rebholz\thanks{School of Mathematical and Statistical Sciences, Clemson University, Clemson, SC 29634, USA ({\tt{rebholz@clemson.edu}}). This work was partially supported by Department of Energy grant DE-SC0025292.}  }

\begin{document}

\maketitle
\begin{abstract}
Anderson acceleration (AA) and nonlinear GMRES (NGMRES) have been successfully applied to improve the efficiency and robustness of many nonlinear solvers, with each being shown to have certain advantages over the other.  There are two differences between the AA and NGMRES algorithms: the extrapolations are different, and AA uses fixed point residuals in the optimization problem while NGMRES uses true nonlinear residuals.  Motivated by recent AA and NGMRES analyses and numerical testing, we consider a blend of AA and NGMRES that uses the AA extrapolation and the true nonlinear residuals for the optimization problem.  To focus our study, we consider the blended method applied to the Picard iteration for the Navier-Stokes equations, but many of the ideas developed herein are extendable to a wide variety of  nonlinear systems and solvers.  We provide a convergence analysis and proof of acceleration of the blended method, provide a detailed numerical comparison of AA, NMGRES and the blended method for several benchmark tests, and develop an effective adaptive depth strategy.

\end{abstract}




\section{Introduction}
There has recently been an explosion of interest in the development and analysis of improved methods for solving nonlinear partial  differential equations and nonlinear optimization problems arising from data science and machine learning. Recent achievements include significant advances in nonlinear acceleration \cite{FaSa09,saad2025acceleration,brezinski2018shanks,wang2021asymptotic}, nonlinear Krylov methods \cite{brown1994convergence,brune2015composing,he2024nltgcr,werner2025nlkrylov,hager2006survey,brown1990hybrid,washio1997krylov,oosterlee2000krylov}, and others \cite{wang2022nonlinear,henson2003multigrid}.

 The purpose of this paper is to study an acceleration method constructed as a blend of Anderson acceleration (AA, Algorithm \ref{alg:AA}) and  nonlinear GMRES (NGMRES, Algorithm \ref{alg:NGMRES}).  
AA and NGMRES are similar in form, but have two key differences: first, their extrapolations are different; and second, AA uses fixed point residuals in its optimization problem while NGMRES uses true nonlinear residuals.  Motivated by recent analyses e.g. \cite{PR21,PR25,HR26} in an attempt to capture better properties of each method, the blended method we consider uses the AA extrapolation and the true nonlinear residuals in its optimization problem.  Such a blend was first proposed this year in \cite{he2026propertiespng}, and was named `modified Anderson acceleration' and denoted as AAg (Algorithm \ref{alg:AAg-cons}).  AAg has not yet been analyzed or tested except for application to linear problems, and herein to focus our study we consider AAg applied to the Picard iteration for the Navier-Stokes equations (NSE).  

To formally define AAg, consider solving $g(u) = 0$ with an iterative scheme 
defined by a fixed point function $q$, which might converge very slowly (or not all all). 
AAg uses the nonlinear residual operator $g(\cdot)$ instead of the fixed-point iteration residual for the least-squares problem.   The main idea of the AAg least squares problem is the same as the least-squares problem defined in \eqref{eq:min-NG} for NGMRES, and we
 aim to accelerate the sequence of
iterates generated by the associated fixed point function $q(\cdot)$. Writing $u^q_j = q(u_j)$, the new update is
based on the linear approximation of the nonlinear operator $g$ 
in the space
$$u_k^q+ \text{span}\{u_k^q-u_{k-1}^q, u_k^q-u_{k-2}^q,\cdots, u_k^q-u_{k-m}^q\}.$$
We approximate $g$
near $u^q_k$
by the Taylor expansion
\begin{align}
g(u_k^q+ \sum_{i=1}^m\xi_i(u_k^q-u_{k-i}^q))&\approx g(u_k^q) + \sum_{i=1}^m \xi_i \frac{\partial g}{\partial u}|_{u_k^q} (u_k^q-u_{k-i}^q) \nonumber\\
&\approx  g(u_k^q) + \sum_{i=1}^m  \xi_i (g(u_k^q)-g(u_{k-i}^q)).\label{eq:NGMRES-expansion}
\end{align}
If we replace $m$ by $m_k=\min \{m, k\}$ and set the new update as $u_{k+1}=u_k^q+ \sum_{i=1}^{m_k}\xi_i(u_k^q-u_{k-i}^q)$, we obtain Algorithm \ref{alg:AAg-cons}. 

AA \cite{Anderson65} is one of the most popular nonlinear acceleration methods for improving fixed-point iterations. The framework of AA is presented in Algorithm \ref{alg:AA}, and its use is straightforward: the current iterate is updated by forming a linear combination of the current and $m$ previous
 fixed-point iterates. At each step, the combination coefficients are obtained by solving a small least-squares problem to minimize the fixed-point residual, and hence the computational cost of solving the least-squares problem is typically negligible compared with evaluating a single fixed-point iteration. The depth $m$ is a prescribed parameter, typically chosen to be small. 
 Over the last 20 years, AA has been shown to accelerate and enable convergence on a very wide range of nonlinear (and linear) problems \cite{WaNi11,PRTX25,FZB20,WSB21_AA,Yang2021_AA,TKSHCP15,AJW17,FAC19_AA,PDZGQL2018_AA,EPRX20,X23,Atanasov2016SteadyStateAA,wang2021asymptotic,bian2022anderson}, and many new variants have been proposed \cite{chen2022composite,he2025generalizedAA,feng2025convergence,barnafi2025two,tang2024anderson}.  The behavior of AA at a single step and its impact on convergence is now at least generally understood, see \cite{PRX19,EPRX20,PR21,PR25}.  Work has also been done regarding the choice of the depth parameter:  A filtering strategy has been proposed in \cite{PR23} to control the condition number of the least-squares matrix, which has been shown to be effective at improving convergence behavior on a range of problems. Also, in \cite{forbes2021anderson, PR21, PR23}, the authors show that using a small depth in the initial iterations and switching to a larger depth later can enhance AA performance. 
 Several open questions remain for AA, including: what can be proven about asymptotic convergence,  what is best way to adaptively choose relaxation and depth parameters, why is the extrapolation used by AA  done in this particular way, and why does the least-squares problem use fixed point residuals instead of the true nonlinear residuals.

NGMRES is another nonlinear solver that has attracted considerable recent attention \cite{washio1997krylov,oosterlee2000krylov,sterck2012nonlinear,sterck2021asymptotic,he2026IMA},  see Algorithm \ref{alg:NGMRES}, and its new developments \cite{he2026convergence,he2025generalized}. 
NGMRES has recently come under study, and understanding it and the blended method studied herein may also provide insights into AA.   Some important recent results for NGMRES include
that under certain conditions full depth NGMRES applied to the Richardson iteration for the linear problem is the same as classical GMRES \cite{GreifHe25NGMRES},  full depth AA and classical GMRES are equivalent under certain conditions for linear problems  \cite{WaNi11}, NGMRES evolves as a nonlinear generalization of GMRES that minimizes the residual for linear systems as does full depth NGMRES applied to preconditioned Richardson iteration and preconditioned GMRES \cite{he2026propertiespng}.

The recent work on NGMRES in \cite{HR26} for the Picard iteration for Navier-Stokes equations (NSE) reveals several important differences between AA and NGMRES for this PDE/iteration, at least comparing to best known AA analysis \cite{PRX19,PR21,PR25}.  First, the convergence analysis for NGMRES is much simpler and straightforward than for AA.  Second, NGMRES was observed to perform better than AA for small constant depths.  Third, the analysis in \cite{HR26} suggests that NGMRES may be better suited for use with superlinearly convergent methods since its residual expansion's higher order terms are of higher order than those of AA.  Fourth, NGMRES convergence analysis has an important feature not present in AA analysis: NGMRES convergence analysis reveals an easily computable quantity that accurately predicts the linear convergence rate at each iteration, and thus by monitoring this quantity it can be determined how much relative effect the higher order terms from the convergence analysis' residual expansion have on the convergence at any given step.  On the other hand, using NGMRES can be more expensive than AA since it is shown in \cite{HR26} that using the Hilbert space dual norm in the optimization problem (instead of $\ell^2$) can be critical for best convergence behavior, but using the dual norm requires an additional linear solve at each NGMRES iteration (for the NSE Picard iteration, this means an extra Stokes solve at each iteration).

%

%

The analysis and computational testing herein provides several important results.  First, the convergence analysis for AAg is short and relatively straightforward, it identifies that the gain of the optimization problem is the mechanism responsible for AAg acceleration, and also reveals a simple way to calculate term that predicts the linear convergence rate at each step (and the numerical tests illustrate its excellent prediction accuracy).  These features resemble those of NGMRES and not AA, and thus they are most likely associated with the use of the true nonlinear residuals in the optimization problem.  Second, the AAg convergence analysis' residual expansion higher order terms are less than quadratic.  Hence this property is most likely associated with the extrapolation.  Third, we provide systematic comparisons of AA, NGMRES and AAg for three benchmark problems and find that for constant depths, AA and AAg had similar convergence behavior, and compared to NGMRES we observe that AA and AAg were generally slightly better for large constant depths but NGMRES was slightly better for small constant depths.  Only for the one test problem, channel flow past a block at Reynolds number 150 (the time dependent regime begins near 50 and so we are solving for a steady solution well into the time dependent regime) was there a dramatic difference: AA and AAg performed much better than NMGRES on this test.

We also show herein that AAg's ability to easily calculate the linear convergence rate and compare it to the actual convergence rate at any step is very useful.  The analysis of AAg, as well as AA and NMGRES, shows that there is a tradeoff in increasing the depth since the convergence rate is the sum of a first order term and higher order terms: increasing the depth decreases the first order term since the optimization problem is done in a higher dimension, but increases the negative effect of the higher order terms (by adding more and older terms to the sum).  Thus AAg can quickly check on the fly whether the higher order terms are significant by checking if the actual convergence rate is the same as the predicted linear convergence rate at any step.  If they are nearly equal, the depth should be increased since AAg will then get a gain from the optimization problem but not experience the negative effects of the higher order terms.  We show this is very effective in reducing the number of iterations needed for convergence.
%
%
%

We now state Algorithms \ref{alg:AA} and \ref{alg:NGMRES} for AA and NGMRES (AAg is given in Section 3).  The notation $\| \cdot \|_{q}$ denotes the norm of the range space of $q$ and $\| \cdot \|_{g}$ denotes the norm of the range space of $g$. In practice, often $\ell^2$ is chosen, but in certain settings using the appropriate function space norms to match the theories provide better results \cite{HR25,HR26}.

	\begin{algorithm}[H] 
		\caption{AA($m$):  Anderson acceleration with depth $m$} \label{alg:AA}
		\begin{algorithmic}[1] 
			\STATE  \textbf{input:} $u_0$, fixed-point operator $q(\cdot)$,  and $m\geq0$ with $m\in\mathbb{N}$
			\FOR {$k=0,1,\cdots$ until convergence }
			\STATE compute 
			\begin{equation}\label{eq:ukp1-AA} 
				u_{k+1} = q(u_k) + \sum_{i=1}^{m_k}\tau_i^{(k)} \left(q(u_k)- q(u_{k-i}) \right),
			\end{equation}
			where $m_k=\min\{k,m\}$ and $\bm{\tau}^{(k)}=\big(\tau_1^{(k)},\cdots, \tau_{m_k}^{(k)}\big)$ is obtained by solving the  least-squares problem
			\begin{equation}\label{eq:min-AA} 
				\min_{\bm{\tau}^{(k)}} \left\|w(u_k)+\sum_{i=1}^{m_k} \tau_i^{(k)} \left(w(u_k)-w(u_{k-i}) \right) \right\|^2_q,
			\end{equation}
            where $w_{k+1}=w(u_k)=q(u_k)-u_k$ is the fixed-point iteration residual.
			\ENDFOR
        \STATE  \textbf{output:} $u_{k+1}$ 
		\end{algorithmic}
	\end{algorithm}

  	\begin{algorithm}[H] 
 		\caption{NGMRES($m$):  Nonlinear GMRES with depth $m$ for solving $g(u)=0$} \label{alg:NGMRES}
 		\begin{algorithmic}[1]
 			\STATE  \textbf{input:} $u_0$, fixed-point operator $q(\cdot)$, and $m\geq0$ with $m\in\mathbb{N}$
 			\FOR {$k=0,1,\cdots$ until convergence}
 			\STATE compute 
 			\begin{equation}\label{eq:ukp1-NG} 
 				u_{k+1} = q(u_k) + \sum_{i=0}^{m_k}\beta_i^{(k)} \left(q(u_k)- u_{k-i} \right),
 			\end{equation}
 			where $m_k=\min\{k,m\}$ and $\bm{\beta}^{(k)}=\big(\beta_0^{(k)},\beta_1^{(k)},\cdots, \beta_{m_k}^{(k)}\big)$ is obtained by solving the  least-squares problem
 			\begin{equation}\label{eq:min-NG} 
 				\min_{\bm{\beta}^{(k)}} \left\| g(q(u_k))+\sum_{i=0}^{m_k} \beta_i^{(k)} \left(g(q(u_k))-g(u_{k-i}) \right) \right\|^2_g.
 			\end{equation}
 			\ENDFOR
            \STATE  \textbf{output:} $u_{k+1}$
 		\end{algorithmic}	 
 	\end{algorithm}

 The remainder of this work is organized as follows. In section \ref{sec:prelim}, we provide a brief introduction to the NSE, its associated Picard iteration, and various properties. In section \ref{sec:analysis}, we present a convergence analysis for AAg with depth one and then with general depth. Numerical results for 2D channel flow past a block, 3D driven cavity, and a 3D stenotic artery model are presented in section \ref{sec:num}, where we show the performance of AAg, AA and NGMRES along with the new adaptive depth strategy. Finally, we draw conclusions in section \ref{sec:con}.

\section{Notation and Preliminaries}\label{sec:prelim}

We consider the domain $\Omega \subset \mathbb{R}^d$ ($d$=2 or 3) to be an open connected set. We use the notation $(\cdot,\cdot)$ and $\|\cdot\|$ for the $L^2(\Omega)$ inner product and norm, respectively.  Other norms will be labeled with subscripts.  

The natural velocity and pressure spaces for the NSE are defined by
	\begin{align*}
			& X:=
		\{v\in H^1\left(\Omega\right): v=0~~\text{on}~ \partial\Omega\},\\
		&Q:=\{q\in {L}^2(\Omega): \int_{\Omega}q\ dx=0\},
		\end{align*}
		and the divergence-free velocity space is given by
		\[
		 V:=
		\{v\in X:  (\nabla \cdot v,q)=0\ \forall q\in Q\}.
		\]
Since $(X,Q)$ satisfy an inf-sup condition \cite{GR86}, analysis can be equivalently performed in the space $V$.  The dual space $X'=H^{-1}(\Omega)$, with norm $\| \cdot \|_{-1}$.  The dual space of $V$ is denoted  by $V'$, and uses norm $\| \cdot \|_{V'}$.  We will abuse notation slightly and use $(\cdot,\cdot)$ also for dual pairings of $V$ and $X$ with their dual spaces.

We recall from \cite{Laytonbook} some well-known properties for the nonlinear form $(v\cdot\nabla w,\chi)$.  First, if $v\in V$ and $w\in X$, then 
\[
(v\cdot\nabla w,w)=0.
\]
Additionally, for $v,w,\chi\in V$, 
\begin{align}
(v\cdot\nabla w,\chi) &\le M \| \nabla v \| \| \nabla w \| \| \nabla \chi \|, \label{bbounds} \\
\| v\cdot\nabla w \|_{V'} &\le M \| \nabla v \| \| \nabla w\|. \label{bbounds2}
\end{align}
This paper works primarily in the space $V$ and its dual space.  The above results, and the rest of the results throughout this paper, would be nearly identical if we worked in the discretely divergence-free subspace $V_h$ derived from an inf-sup stable velocity-pressure pair $(X_h,Q_h)\subset (X,Q)$, provided that a skew-symmetric form of the nonlinearity is used if divergence-free finite elements were not used, e.g. skew-symmetric form, rotational form or EMAC \cite{JLMNR17,CHOR17}.

\subsection{NSE preliminaries}
The steady NSE are given by
\begin{equation}\label{NS}
		\left\{\begin{aligned}
			-\nu \Delta u+u\cdot\nabla u+ \nabla p&={f} \quad \text{in}~\Omega,\\
			\nabla\cdot {u}&=0\quad \text{in}~\Omega,\\
			{u}&=0 \quad \text{on}~~\partial\Omega,
		\end{aligned}\right.
\end{equation}
where $u$ is the velocity and $p$ is the pressure, $\nu>0$ is the kinematic viscosity and ${f}$ is a given external forcing. The Reynolds number $Re=O(\frac{1}{\nu})$ is another physical constant often used to describe the complexity of a flow.  We consider the system \eqref{NS} equipped with homogeneous Dirichlet boundary conditions for simplicity of analysis, but our results are extendable to nonhomogeneous mixed Dirichlet/Neumann boundary conditions and to a time dependent NSE scheme at a fixed time step.  

The weak form of \eqref{NS} is: Find $u\in V$ such that
\begin{equation}
\nu(\nabla u,\nabla v) + (u\cdot\nabla u,v) = (f,v)\ \ \forall v\in V.\label{NSw}
\end{equation}
It is known that solutions to \eqref{NSw} exist for any $\nu>0$ and $f\in H^{-1}(\Omega)$ \cite{Laytonbook}, and are bounded by 
\begin{equation}
\| \nabla u \| \le \nu^{-1} \| f \|_{-1}. \label{nsstab}
\end{equation}
Under a small data sufficient condition $\kappa:=M\nu^{-2} \|f \|_{H^{-1}}<1$, solutions are known to be unique \cite{GR86,Laytonbook,temam}.  However, for $\kappa$ large enough, \eqref{NSw} can have multiple solutions \cite{Laytonbook}.

\subsection{Picard and nonlinear residual preliminaries}

The Picard iteration for solving \eqref{NSw} is: Given $\omega_k\in V$, find $q(\omega_k)\in V$ such that 
\begin{equation}
\nu(\nabla q(\omega_k),\nabla v) + (\omega_k \cdot\nabla q(\omega_k),v) = ( f,v )\ \forall v\in V. \label{NSEPic}
\end{equation}
It is shown in \cite{PRX19} that the solution operator $q:V\rightarrow V$ associated with \eqref{NSEPic} is well-defined and is uniformly bounded:
\begin{equation}
\| \nabla  q(\omega_k) \|  \le \nu^{-1} \| f\|_{-1}. \label{Picbound}
\end{equation}  
Thus, we can define the fixed-point Picard iteration as $\omega_{k+1}=q(\omega_k)$.  In AAg (and AA), Picard is used as a first step and then a second step is used as a type of extrapolation or correction.  Thus we will also use the notation $\tilde \omega_{k+1}=q(\omega_k)$ for a Picard iterate in that context.  The Picard fixed-point iteration residual satisfies \cite{GR86}
\begin{equation}
 \| \nabla (q (\omega_{k})-\omega_{k} ) \| \le \kappa \| \nabla ( q(\omega_{k-1})-\omega_{k-1}) \|. \label{Picresid}
\end{equation}

Equation \eqref{Picresid} implies that $\kappa<1$ yields a contractive Picard iteration that will converge linearly with rate $\kappa$ to the unique weak solution.  For sufficiently large $\kappa$, it is known that the Picard iteration will diverge \cite{PR25}.  Herein, we assume that the Picard iteration is contractive by assuming that $\kappa<1$.  This assumption is needed in our analysis for the higher order terms, although in our numerical tests the we use $\kappa \gg 1$, i.e. far above the range where Picard is contractive.  This restriction on $\kappa$ is common for steady NSE numerical scheme analysis, as not surprisingly analysis gets much more difficult when solutions are not necessarily unique.  The issue needed to overcome this assumption is a generalization of Lemma \ref{elem} which does not assume $\kappa<1$, which has so far eluded the authors.

We also consider the NSE nonlinear residual operator, which is defined by $g:V \rightarrow V'$ satisfying for all $v\in V$ and $\chi\in V$, 
\begin{equation}
( g(v),\chi ) =   \nu (\nabla  v,\nabla \chi) + (v\cdot\nabla v,\chi) - (f,\chi). \label{gdefn}
\end{equation}
The $V'$ norm of $g(v)$ is thus defined by
\begin{equation}\label{eq:def-V'norm}
\|g(v)\|_{V'} = \max_{0\neq \chi \in V} \frac{(g(v),\chi)}{\|\nabla \chi\|}.
\end{equation}
Due to $v\in V$, the $v\cdot\nabla v$ term cannot be expected to be in $L^2(\Omega)$, and thus $V'$ is the appropriate norm for measuring $g(v)$.  

 \section{Convergence analysis of AAg}\label{sec:analysis}

We now consider convergence analysis of AAg applied to the Picard iteration for the NSE.  Hence in what follows, $g$ is the nonlinear residual defined by \eqref{gdefn}, and $q$ is the NSE Picard iteration defined by \eqref{NSEPic}. In this setting with Picard for the NSE, $\| \cdot \|_{V'}$ is the $g$-norm used in the
the least-squares problem given by \eqref{eq:min-AAg-cons} in Algorithm \ref{alg:AAg-cons}, which expresses the combination coefficients in an equivalent constraint-based format. For simplicity, since we do a single step analysis, we omit the superscript that appears in the coefficients $\alpha_{k+1-i}^{(k)}$ in the following analysis.

   \begin{algorithm}[H] 
		\caption{AAg($m$): Anderson acceleration with depth $m$ for $g(u)=0$ using true nonlinear residuals} \label{alg:AAg-cons}
		\begin{algorithmic}[1] 
			\STATE  \textbf{input:} $u_0$, fixed-point operator $q(\cdot)$,  and $m\geq0$ with $m\in\mathbb{N}$
			\FOR {$k=0,1,\cdots$ until convergence }
			\STATE compute 
			\begin{equation*} 
				u_{k+1} = \sum_{i=0}^{m_k}\alpha_{k+1-i}^{(k)} q(u_{k-i}),
			\end{equation*}
			where $m_k=\min\{k,m\}$ and $\{\alpha_{k+1-i}^{(k)}\}_{i=0}^{m_k}$ are obtained by solving the  least-squares problem
			\begin{equation}\label{eq:min-AAg-cons}
	\min_{\sum_{i=k-m_k+1}^{k+1}\alpha_i^{(k)}=1} \left\|\alpha_{k-m_k+1}^{(k)}g(q(u_{k-m_k}))+ \alpha_{k-m_k+2}^{(k)} g(q(u_{k-m_k+1}))\cdots + \alpha_{k+1}^{(k)} g(q(u_k)) \right\|^2_{V'}.
			\end{equation}
			\ENDFOR
        \STATE  \textbf{output:} $u_{k+1}$
		\end{algorithmic}
	\end{algorithm}

For notational simplicity, it is very helpful in our analysis to denote $\tilde u_{k+1}:=q(u_k)$ and $e_k=u_k - u_{k-1}$.  It is also useful in our analysis to define the Picard fixed-point residual by
\[
w_{k+1}  = q(u_k) - u_k = \tilde u_{k+1} - u_k.
\]
The following bounds were proven in \cite{HR26} regarding the Picard fixed point residual and the NSE nonlinear residual:
\begin{align}
    \| \nabla w_{k+1} \| & \le \nu^{-1} \| g(u_k) \|_{V'}, \label{wg1} \\
    \| g(\tilde u_{k+1}) \|_{V'}  & \le \kappa \| g( u_{k}) \|_{V'}. \label{wg2}
\end{align}

Before proceeding with a convergence analysis, we recall an important result: the lemma from \cite[Section~3.2]{HR26}. 
This lemma relates the $V$ norm of the difference of iterates for any extrapolation method where $q$ is the Picard iteration and $g$ is defined by \eqref{gdefn}.  In \cite{HR26} it is stated that this result holds for the difference in NGMRES iterates.  However, it actually holds for any two functions in $V$, with the exact same proof (assuming $q$ and $g$ are defined as above).  We give the proof in the Appendix for completeness.

\begin{lemma}\label{elem} 
Provided $\kappa<1$, the difference between general functions $\hat u,\hat v\in V$ is bounded by
\begin{equation}
\| \nabla (\hat u-\hat v) \| \le \frac{\nu^{-1}}{1-\kappa} ( \| g(\hat u)\|_{V'} + \| g(\hat v)\|_{V'}). \label{ebound1}
\end{equation}
\end{lemma}

It will also be useful to define the quantity $\tilde e_k=\tilde u_k - \tilde u_{k-1}$.  Since for all $v\in V$,
\begin{align*}
\nu(\nabla \tilde u_k,\nabla v) + (u_{k-1}\cdot\nabla \tilde u_k,v) = (f,v),\\
\nu(\nabla \tilde u_{k-1},\nabla v) + (u_{k-2}\cdot\nabla \tilde u_{k-1},v) = (f,v),
\end{align*}
we have that
\[
\nu (\nabla \tilde e_k,\nabla v) + (e_{k-1}\cdot\nabla \tilde u_k,v) + (u_{k-2}\cdot\nabla \tilde e_k,v)=0.
\]
Hence with $v=e_k$, using \eqref{bbounds}, \eqref{Picbound} and the definition of $\kappa$ we get that
\begin{equation}
\| \nabla \tilde e_k \| \le M \nu^{-1} \| \nabla e_{k-1} \| \| \nabla \tilde u_k \| \implies \| \nabla \tilde e_k \| \le \kappa \| \nabla e_{k-1} \|. \label{etbound}
\end{equation}
This will be useful bound below.

We now prove convergence estimates for AAg. Under the assumption that $\kappa<1$ so that the Picard iteration is contractive, we prove that AAg accelerates convergence by scaling the linear convergence rate by the gain of the optimization problem, which is defined by
\begin{align*}
\theta_{k+1} &=  \frac{ \| \sum_{j=k-m_k+1}^{k+1}  \alpha_j g(\tilde u_j) \|_{V'} } { \| g(\tilde u_{k+1}) \|_{V'}  }.
\end{align*}
The numerator comes from the optimization problem, and the denominator comes from not using acceleration but just the Picard iteration, which is equivalent to picking the parameters $[1,0,...,0]$.  Since $[1,0,...,0]$ is in the admissible set of parameters, we have that $\theta\le 1$, but $\theta=1$ is very unlikely.

Further, we give an estimate of the linear convergence rate of $\| g(u_k) \|_{V'}$ at each iteration:
\begin{align*}
\gamma_{k+1} &=  \frac{ \| \sum_{j=k-m_k+1}^{k+1}  \alpha_j g(\tilde u_j) \|_{V'} } { \| g(u_{k}) \|_{V'}  }.
\end{align*}
Our numerical tests show $\gamma_{k+1}$ gives a very sharp estimate of the contraction ratio at each iteration, once the residual is moderately reduced so the higher order terms become negligible.  We also show in our numerical tests that $\gamma_{k+1}$ can be used to adaptively choose the depth at each step of AAg.  

The AAg convergence theory for general $m$ is long and somewhat tedious, and so we first prove the result for $m=1$ so that reader can have some intuition of the proof structure before we prove the general $m$ result.
 
\subsection{AAg convergence for $m=1$}
We begin by establishing two identities for the $m=1$ case. Recall that
$$u_{k+1}=\alpha_{k+1}\tilde u_{k+1} +\alpha_k \tilde u_k,$$
where $\alpha_{k+1}+\alpha_k=1$. First, expanding $u_{k+1}-\tilde u_{k+1}$ and using the definition of $u_{k+1}$ yields 
\begin{align}
u_{k+1}-\tilde u_{k+1} & =-(1-\alpha_{k+1}) \tilde u_{k+1} + \alpha_k \tilde u_k = -\alpha_k \tilde e_{k+1}. \label{I1}
\end{align}
Similarly, expanding $u_{k+1}-\tilde u_{k}$ gives
\begin{align}
u_{k+1}-\tilde u_{k} & = \alpha_{k+1} \tilde u_{k+1} + (\alpha_k-1) \tilde u_k = \alpha_{k+1} \tilde e_{k+1}. \label{I2}
\end{align}

\begin{theorem} Consider AAg with $m=1$ presented in Algorithm \ref{alg:AAg-cons}. Assume that $\max_k|\alpha_k|\leq \bar{\alpha}$ for some positive $\bar{\alpha}$. 
If $\kappa<1$, then for $k>1$,
\begin{align*}
\| g(u_{k+1})\|_{V'} &\le \theta_{k+1} \kappa \| g(u_k)  \|_{V'}
+ \frac{2\bar\alpha^2 \kappa^2}{\nu^2(1-\kappa)^2} \left( \| g(u_k) \|_{V'}^2 + \| g(u_{k-1}) \|_{V'}^2  \right), \\
\| g(u_{k+1})\|_{V'} &\le \gamma_{k+1} \|g(u_{k})  \|_{V'}
+ \frac{2\bar\alpha^2 \kappa^2}{\nu^2(1-\kappa)^2} \left( \| g(u_k) \|_{V'}^2 + \| g(u_{k-1}) \|_{V'}^2  \right).
\end{align*}
\end{theorem}
\begin{remark}
Since $\kappa$ is the linear convergence rate of Picard iteration, the quantity $\theta_{k+1}$ represents the convergence acceleration provided by AAg.  Since in practice $\kappa$ is hard to know precisely, the quantity $\gamma_{k+1}$ represents a sharper estimate of the linear convergence rate at each iteration. Note that it is consistent with \eqref{eq:NGMRES-expansion} that the $(k+1)$st residual is bounded by the linear part of $k$th residual and high-order part of $k$th and $(k-1)$st residuals.
\end{remark}

\begin{proof}
To reduce on notation, the following two equality strings will be understood to hold in the weak sense when tested with an arbitrary function from $V$.

We begin the proof by developing an identity for $g(u_{k+1})$ through expanding $u_{k+1}$ using its definition and that $\alpha_k + \alpha_{k+1}=1$ to obtain
\begin{align*}
g(u_{k+1}) &= u_{k+1} \cdot\nabla u_{k+1} - \nu\Delta u_{k+1} - f \\ 
&= \alpha_{k+1} \left( \tilde u_{k+1} \cdot\nabla u_{k+1} - \nu\Delta \tilde u_{k+1} - f\right) +
\alpha_k \left( \tilde u_{k} \cdot\nabla u_{k+1} - \nu\Delta \tilde u_k - f\right)  \\ 
& = \alpha_{k+1} \left( g(\tilde u_{k+1}) + \tilde u_{k+1}\cdot\nabla (u_{k+1}-\tilde u_{k+1} ) \right)
+ \alpha_{k} \left( g(\tilde u_{k}) + \tilde u_{k}\cdot\nabla (u_{k+1}-\tilde u_{k}) \right).
\end{align*}
Next, regrouping terms and using \eqref{I1}-\eqref{I2}, we get that
\begin{align*}
g(u_{k+1}) &= \left( \alpha_{k+1} g(\tilde u_{k+1}) + \alpha_{k} g(\tilde u_{k}) \right)
+  \alpha_{k+1} \tilde u_{k+1}\cdot\nabla (u_{k+1}-\tilde u_{k+1} )
+ \alpha_{k}  \tilde u_{k}\cdot\nabla (u_{k+1}-\tilde u_{k}) \\
&= \left( \alpha_{k+1} g(\tilde u_{k+1}) + \alpha_{k} g(\tilde u_{k}) \right)
- \alpha_{k+1}\alpha_k \tilde u_{k+1}\cdot\nabla \tilde e_{k+1} + \alpha_k \alpha_{k+1} \tilde u_{k}\cdot\nabla \tilde e_{k+1} \\
&= \left( \alpha_{k+1} g(\tilde u_{k+1}) + \alpha_{k} g(\tilde u_{k}) \right)
- \alpha_{k}\alpha_{k+1} \tilde e_{k+1}\cdot\nabla \tilde e_{k+1} .
\end{align*}
Taking $V'$ norms of both sides, using the triangle inequality, the bound on $\alpha_j$'s and \eqref{etbound}  gives 
\begin{align*}
\| g(u_{k+1})\|_{V'} 
&\le
\| \alpha_{k+1} g(\tilde u_{k+1}) + \alpha_{k} g(\tilde u_{k})  \|_{V'}
+ \bar\alpha^2 \| \nabla \tilde e_{k+1} \|^2\\
&\le
\| \alpha_{k+1} g(\tilde u_{k+1}) + \alpha_{k} g(\tilde u_{k})  \|_{V'}
+ \bar\alpha^2 \kappa^2 \| \nabla e_{k} \|^2.
\end{align*}
Finally, using \eqref{ebound1}, we obtain the estimate
\begin{align*}
\| g(u_{k+1})\|_{V'} 
&\le
\| \alpha_{k+1} g(\tilde u_{k+1}) + \alpha_{k} g(\tilde u_{k})  \|_{V'}
+ \frac{\bar\alpha^2 \kappa^2}{\nu^2(1-\kappa)^2} \left( \| g(u_k) \|_{V'} + \| g(u_{k-1}) \|_{V'}  \right)^2.
\end{align*}
From here, inserting the definitions of $\theta_{k+1}$ and $\gamma_{k+1}$ for the lower order term along with Young's inequality for the higher order terms gives the two estimates
\begin{align*}
\| g(u_{k+1})\|_{V'} 
&\le
\theta_{k+1} \| g(\tilde u_{k+1})   \|_{V'}
+ \frac{\bar\alpha^2 \kappa^2}{\nu^2(1-\kappa)^2} \left( \| g(u_k) \|_{V'} + \| g(u_{k-1}) \|_{V'}  \right)^2, \\
\| g(u_{k+1})\|_{V'} 
&\le
\gamma_{k+1} \|g( u_{k})  \|_{V'}
+ \frac{\bar\alpha^2 \kappa^2}{\nu^2(1-\kappa)^2} \left( \| g(u_k) \|_{V'} + \| g(u_{k-1}) \|_{V'}  \right)^2.
\end{align*}
Lastly, using \eqref{wg2}
finishes the proof.
\end{proof}

\subsection{AAg convergence for general $m$}

We now prove a convergence estimate for AAg in the case of general $m$.  The strategy is the same as for the $m=1$ case: first we write $g(u_{k+1})$ in terms of $\sum_{j=k-m+1}^{k+1} \alpha_j g(\tilde u_j)$, and then bound the remaining terms.  We first establish a lemma useful for the main theorem to follow.

\begin{lemma}\label{ilemma}
Consider AAg with $m>1$ presented in Algorithm \ref{alg:AAg-cons}.  For $k>m$, the following identities hold:
\begin{align*}
u_{k+1} - \tilde u_{k+1} & = -(1-\alpha_{k+1}) \tilde e_{k+1} - (1-\alpha_{k+1}-\alpha_k) \tilde e_{k} - ...- \alpha_{k-m+1} \tilde e_{k-m+2}, \\
u_{k+1} - \tilde u_{k} &=  \alpha_{k+1} \tilde e_{k+1} - (1-\alpha_{k+1} - \alpha_k) \tilde e_k - (1-\alpha_{k+1} - \alpha_k - \alpha_{k-1}) \tilde e_{k-1} \\
& \ \ \ \ -  ...-  \alpha_{k-m+1} \tilde e_{k-m+2}, \\
u_{k+1} - \tilde u_{k-1} & = \alpha_{k+1} \tilde e_{k+1} +  (\alpha_{k+1} + \alpha_k) \tilde e_k - (1-\alpha_{k+1}-\alpha_{k}- \alpha_{k-1}) \tilde e_{k-1} \\
& \ \ \ \ - ...- \alpha_{k-m+1} \tilde e_{k-m+2}, \\
\mbox{for $j<k-1$:} \ \ \ u_{k+1} - \tilde u_j &=\alpha_{k+1} \tilde e_{k+1} +  (\alpha_{k+1} + \alpha_k) \tilde e_k + ... +  (\alpha_{k+1} + \alpha_k + ... +\alpha_{j+1}) \tilde e_{j+1} \\
& \ \ \ \ (1 - \alpha_{k+1} - ... - \alpha_j)\tilde e_j - ... - \alpha_{k-m+1}\tilde e_{k-m+2}.
\end{align*}
\end{lemma}
\begin{proof}
This proof follows from identities, and is given in the Appendix.
\end{proof}

The general $m$ AAg convergence theorem can now be stated as follows.
\begin{theorem}\label{thmm} Consider AAg with $m>1$ presented in Algorithm \ref{alg:AAg-cons}. Assume that $\max_k|\alpha_k|\leq \bar{\alpha}$ for some positive $\bar{\alpha}$.
If $\kappa<1$, then for $k>m$:
$$\| g(u_{k+1}) \|_{V'} \le
\min \{ \theta_{k+1}\kappa,\ \gamma_{k+1} \} \| g(u_{k}) \|_{V'}  + 4M (m+1) m^3 \bar \alpha^2 \frac{\kappa^2}{\nu^{2} (1-\kappa)^2} \sum_{i=0}^m \| g(u_{k-i}) \|_{V'}^2.$$
Here $M$ satisfies \eqref{bbounds2}.
\end{theorem}
Before presenting the proof, we make the following two remarks to better understand Theorem \ref{thmm}.
\begin{remark} For $1<k\leq m$, the result in Theorem \ref{thmm} holds after replacing $m$ with $m_k=\min\{m,k\}$ on the right-hand side on the above inequality.
\end{remark}
\begin{remark}
Since $\kappa$ is the linear convergence rate of Picard, the quantity $\theta_{k+1}$ represents the convergence acceleration provided by AAg.  Since in practice $\kappa$ is hard to know precisely, the quantity $\gamma_{k+1}$ represents a sharper estimate of the linear convergence rate at each iteration.  We show in our numerical tests that $\gamma_{k+1}$ is a sharp estimate of the linear convergence rate, and indeed sharp enough that it can be used to develop strategies for adaptive $m$ at each iteration. 
\end{remark}

\begin{proof}
Note that for $k>m$, $m_k=\min\{m,k\}=m$. We establish using a long string of identities in \eqref{Rexpansion0} and \eqref{Rexpansion} in the Appendix that 
\begin{align}
g(u_{k+1}) & = 
\alpha_{k+1} g(\tilde u_{k+1}) + \alpha_{k} g(\tilde u_{k}) + ...  
+ \alpha_{k-m+1} g( \tilde u_{k-m+1} ) 
 +  \sum_{j=k-m+2}^{k+1} \sum_{i=j}^{k+1} \alpha_{i}\tilde e_j \cdot\nabla (u_{k+1} - \tilde u_{i}), \label{III1}
 \end{align}
 where equality is meant in the weak sense, i.e. equality holds when the equation is tested with an arbitrary function from $V$.  Using Lemma \ref{ilemma}, up to coefficients all higher order terms take the form $\tilde e_j \cdot \nabla \tilde e_i$.  Thus using \eqref{bbounds2}, \eqref{etbound} and \eqref{ebound1}, we have that
\begin{align}
 \sum_{j=k-m+2}^{k+1} &\sum_{i=j}^{k+1} \alpha_{i}\tilde e_j \cdot\nabla (u_{k+1} - \tilde u_{i}) \nonumber \\
& \le 4M (m+1) m \bar \alpha^2 \sum_{i,j=k-m+2}^{k+1} \| \nabla \tilde e_j \| \| \nabla \tilde e_i \| \nonumber \\
&  \le 4M (m+1) m^3  \frac{\bar \alpha^2\kappa^2}{\nu^{2} (1-\kappa)^2} \sum_{j=k-m}^k \| g(u_j) \|_{V'}^2. \label{gen1}
\end{align}
Now taking $V'$ norms of \eqref{III1} and using \eqref{gen1} provides
\begin{multline*}
\| g(u_{k+1}) \|_{V'}
\le
\| \alpha_{k+1} g(\tilde u_{k+1}) + \alpha_{k} g(\tilde u_{k}) + ...  
+ \alpha_{k-m+1} g( \tilde u_{k-m+1} ) \|_{V'} \\
+4 M (m+1) m^3 \bar \alpha^2 \frac{\kappa^2}{\nu^{2} (1-\kappa)^2} \sum_{j=k-m}^k \| g(u_j) \|_{V'}^2.
\end{multline*}
Now using the definitions of $\theta_{k+1}$ and $\gamma_{k+1}$, we obtain
\begin{align}
\| g(u_{k+1}) \|_{V'}
& \le
\theta_{k+1} \| g(\tilde u_{k+1}) \|_{V'}
 + 4M (m+1) m^3 \bar \alpha^2 \frac{\kappa^2}{\nu^{2} (1-\kappa)^2} 
\bigg(  \| g(u_k) \|_{V'}^2 + \| g(u_{k-1}) \|_{V'}^2 + ... +  \| g(u_{k-m}) \|_{V'}^2 \bigg), \label{gen2}
\end{align}
and
\begin{align*}
\| g(u_{k+1}) \|_{V'}
& \le
\gamma_{k+1} \| g( u_{k}) \|_{V'} + 4M (m+1) m^3 \bar \alpha^2 \frac{\kappa^2}{\nu^{2} (1-\kappa)^2} 
\bigg(  \| g(u_k) \|_{V'}^2 + \| g(u_{k-1}) \|_{V'}^2 + ... +  \| g(u_{k-m}) \|_{V'}^2 \bigg).
\end{align*}
Finally, using \eqref{wg2} in \eqref{gen2} completes the proof.
\end{proof}

\section{Numerical Experiments}\label{sec:num}

We now present results of several numerical tests to illustrate the AAg theory, compare convergence behavior of AAg,  AA and NGMRES, and discuss an adaptive depth strategy for AAg.  The theorems presented in this work use the $V'$ norm, and we compute this norm in our computations using
\[
\| \phi_h \|_{V'}^2 := (\phi_h,A_h^{-1} \phi_h),
\]
where $A_h$ is the discrete Stokes operator with boundary conditions enforced in the matrix (i.e. remove degrees of freedom (dof) associated with Dirichlet velocity boundary conditions).  As shown in \cite{HR26} for NGMRES and \cite{HR25} for AA, using the norm in the optimization problem that is consistent with the theory is important for convergence behavior guaranteed by the theory (an obvious statement for a numerical analyst).  Using $\ell^2$ instead as the optimization norm is more efficient and can often also provide good convergence results sometimes, but can also cause failure \cite{HR25,HR26}; hence it is safest to compute with the norms consistent with the theory.


For the NSE Picard iteration, on the first step we solve directly for $(u_1,p_1)$ but on subsequent steps we solve directly for the updates $\delta_{k+1}^u = u_{k+1} - u_k$ and $\delta_{k+1}^p = p_{k+1} - p_k$, and then recover $u_{k+1}$ and $p_{k+1}$.  While these approaches to the Picard iteration are mathematically equivalent, when using inexact linear solvers it is advantageous to solve for the updates.  For example, solving for the updates allows for a divergence error of near machine precision in our tests, whereas if we solved directly for $(u_{k+1},p_{k+1})$ with GMRES tolerance of $10^{-8}$ we would obtain only $\| \nabla \cdot u_{k+1} \|\approx 10^{-3}$ for the 3D cavity tests.

In this section, we first present the AAg convergence results and results for $\gamma_k$ as a predictor for the convergence rate, then compare the performance of AAg, AA and NGMRES, and finally propose and test an adaptive depth strategy for AAg. 

\subsection{Test problem setup}

We now give descriptions of the test problems used in our numerical experiment: 2D channel flow past a block, 3D driven cavity, and a 3D stenotic artery model.

\subsubsection{Channel flow past a cylinder}

For our first test problem, we use the classical 2D channel flow past a block benchmark from e.g. 
\cite{TGO15, R97, SDN99,HRV24}.  The domain is a $2.2\times0.41$ rectangular channel with a square of side length $0.1$ centered at $(0.2, 0.2)$.  No-slip velocity is enforced on the walls and block, and at the inflow and outflow we enforce  
\begin{align}
	u(0,y)= \left( \begin{array}{c} \frac{6}{0.41^2}y(0.41-y) \\ 0 \end{array}\right).
\end{align}
There is no external forcing ($f=0$), and we use varying $Re$=$\frac{UL}{\nu}$=100, 150 and 200, which are calculated using length scale $L=0.1$ as the block width, and $U=1$. We note that the NSE with $Re>50$ is known to produce periodic-in-time solutions \cite{HRV24,OR25}, but we are solving for steady state solutions at these $Re$ (recall steady solutions exist for any $Re$ \cite{Laytonbook}).  A solution from our computations with $Re$=100 is shown in Figure \ref{cylpic}.

\begin{figure}[H]
  \centering
       \includegraphics[scale=0.11]{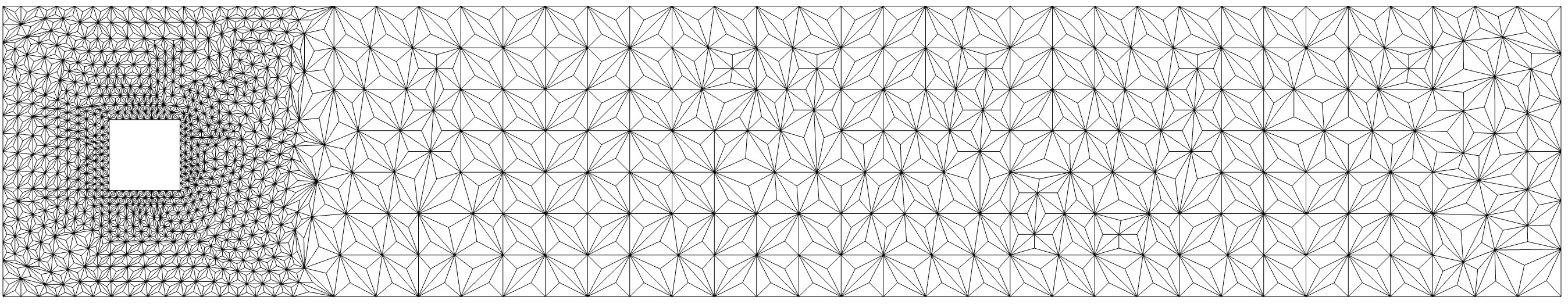}\\
        \includegraphics[scale=0.34]{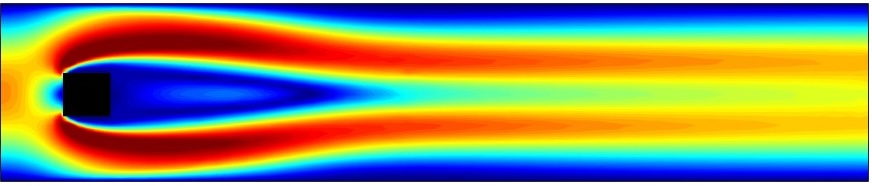} 
  \caption{\label{cylpic} Shown above is a sample mesh for 2D channel flow past a block (top) and $Re$=100 solution shown as speed contours.}
\end{figure}

For the discretization, $(P_2, P_1^{disc})$ Scott-Vogelius (SV) elements are used on barycenter refined Delaunay meshes of the domain.  Computations were done on several meshes, with results nearly identical across meshes (which is expected, as the convergence results are mesh independent and other AA and NGMRES tests have always shown mesh independence \cite{PR25,HR26}.  The mesh used for the results below provided 120K velocity and 89K pressure dof.  Figure \ref{cylpic} shows a coarse mesh of the domain, significantly coarser than what is used in our tests.  The nonlinear solver tolerance was taken to be $10^{-8}$ in the $V'$ norm, and the linear solvers for these 2D tests used a direct solver.

\subsubsection{3D driven cavity}

\begin{figure}[h]
\center
$Re$=1000 \\
\includegraphics[width = .9\textwidth, height=.27\textwidth,viewport=120 0 1180 340, clip]{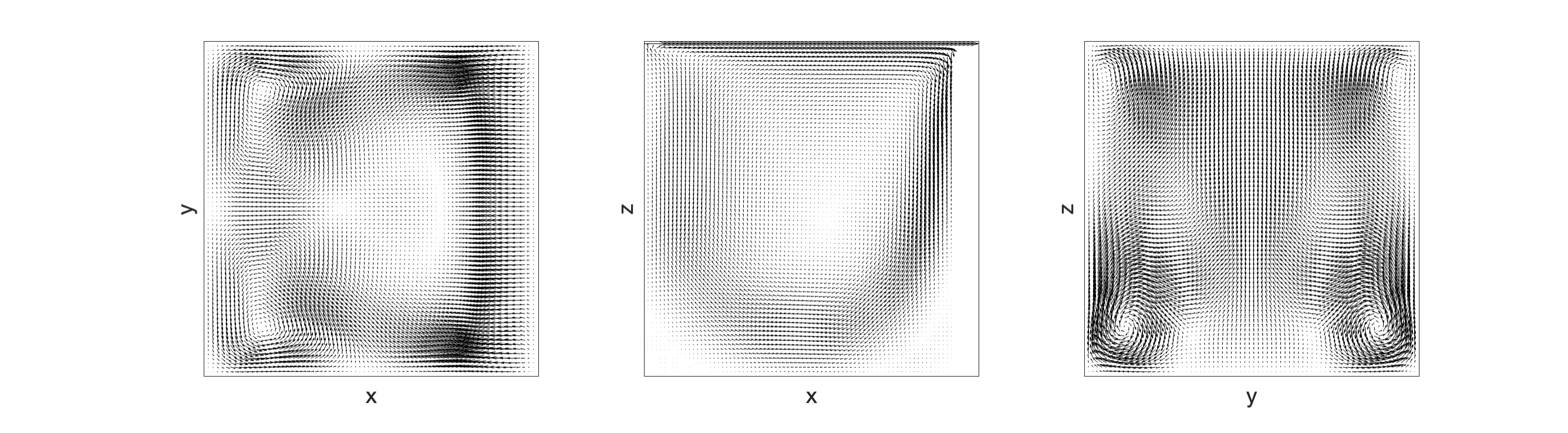}  \\
\caption{\label{cav3d} The plot above shows the $Re$=1000 3D driven cavity velocity solution displayed as midsliceplanes of velocity.}
\end{figure}

The 3D driven cavity is used for our second benchmark test.  This problem is analogous to the well studied 2D driven cavity, where $f=0$, $\Omega=(0,1)^3$, and 
homogeneous Dirichlet boundary conditions for velocity are enforced on the walls but $[1, 0, 0]^T$ is enforced on the `lid'.  We compute with varying $Re=\nu^{-1}$, and midspliceplane velocity vector solutions for $Re$=1000 are shown in Figure \ref{cav3d}.

The mesh is constructed as follows: first we create a $ \mathcal{M}^3$ grid of boxes from Chebychev points on $[0, 1]$,  then refine each box into 6 tetrahedra, and finally barycenter refine each of those tetrahedra into 4 tetrahedra.  Computations used $\mathcal{M}=$11 and 13, which provided 796K and 1.3M total dof, respectively, when equipped with $(P_3,P_2^{disc})$ SV elements, and no difference was observed for the convergence or solutions on the two meshes.  Hence results are shown only for the fine mesh below.  Our computed solutions were found to match the literature well \cite{WB02}, see Figure \ref{cav3d}.  The nonlinear solver tolerance was $10^{-7}$ in the $V'$ norm.  The linear solves used an augmented Lagrangian preconditioning essentially following \cite{benzi,HR13,BL12} (using grad-div stabilization and preconditioning the Schur complement with the pressure mass matrix) with GMRES as the outer solver and direct inner solves of the velocity submatrix.  The outer linear solver tolerance was set as $10^{-8}$, and convergence was generally reached within 2 to 3 iterations.

\subsubsection{Flow in a 3D stenotic artery model}

\begin{figure}[h!]
\center
\includegraphics[width = .8\textwidth, height=.23\textwidth,viewport=0 0 1200 300, clip]{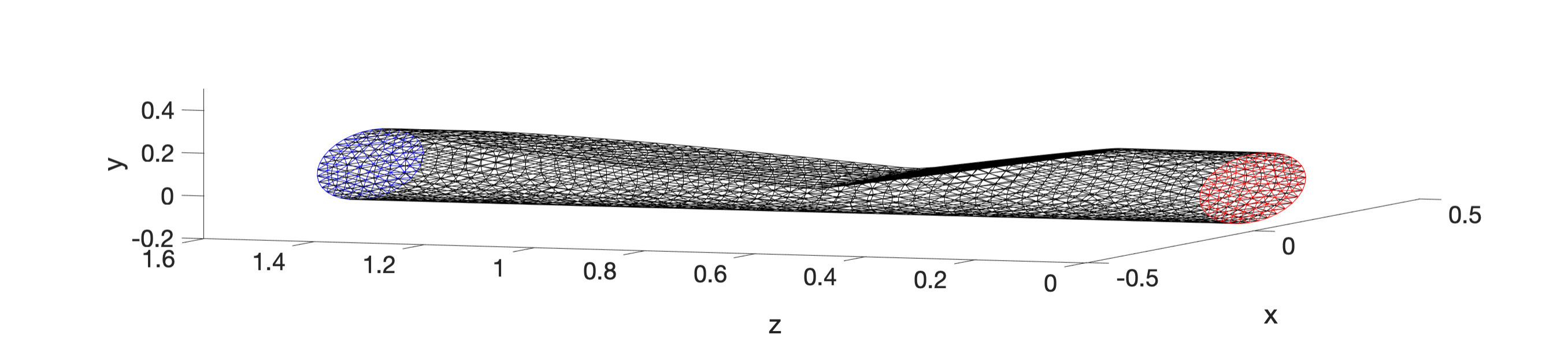}
\caption{\label{arterymesh} Shown above is the artery mesh (before the barycenter refinement is applied) restricted to the surface. }
\end{figure}

\begin{figure}[h!]
\center
\includegraphics[width = .8\textwidth, height=.23\textwidth,viewport=50 0 1200 300, clip]{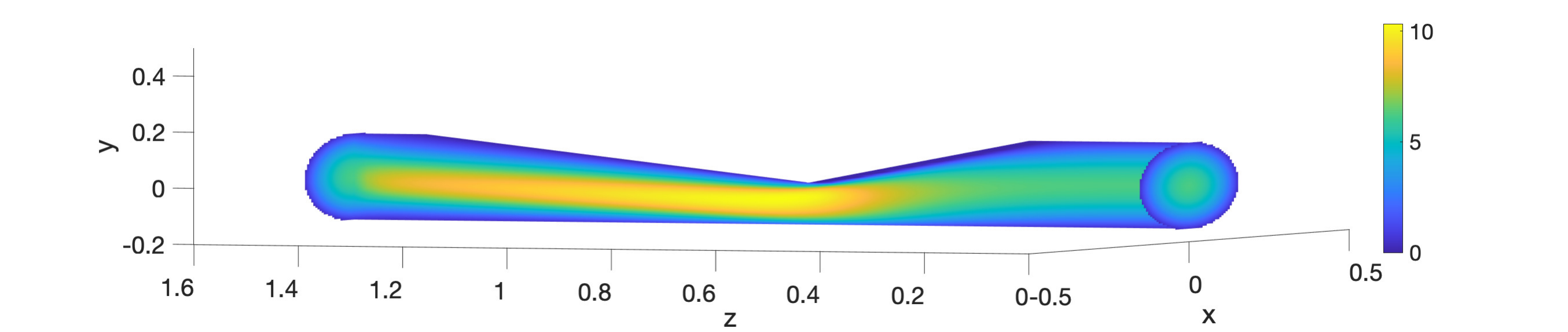}  
\caption{\label{artery100} Shown above is speed contour slices for the $\nu=\frac{1}{500}$ solution.}
\end{figure}

For our last test problem, we use the stenotic artery model from \cite{TKLV24}.  This application benchmark has length units given in $cm$ and time units in seconds. The vessel domain is created from a 3D deformed cylinder with diameter $d=0.16$ and length $l=1.6$.  We enforce no-slip velocity on the artery walls, and parabolic and outflow using a maximum velocity at the vessel center of 6.22.
The viscosity is taken to be $\nu=0.002$,
and no forcing is used. 

The domain is discretized first from a regular tetrehedralization (thanks to the authors of \cite{TKLV24} for providing the mesh!), which is shown as a surface mesh in Figure \ref{arterymesh}.  This mesh is then barycenter refined, which gives 52,912 total tetrahedra, and is then equipped with $(P_3,P_2^{disc})$ SV elements, yielding 1.33M total dof.  A plot of the solution found on this mesh using AAg-Picard is shown in Figure \ref{artery100}.  The same linear solver used for the 3D driven cavity is used on this problem.  

\subsection{Tests for AAg convergence and $\gamma_{k}$}

\begin{figure}[h!]
\center
\includegraphics[width = .32\textwidth, height=.23\textwidth,viewport=0 0 530 420, clip]{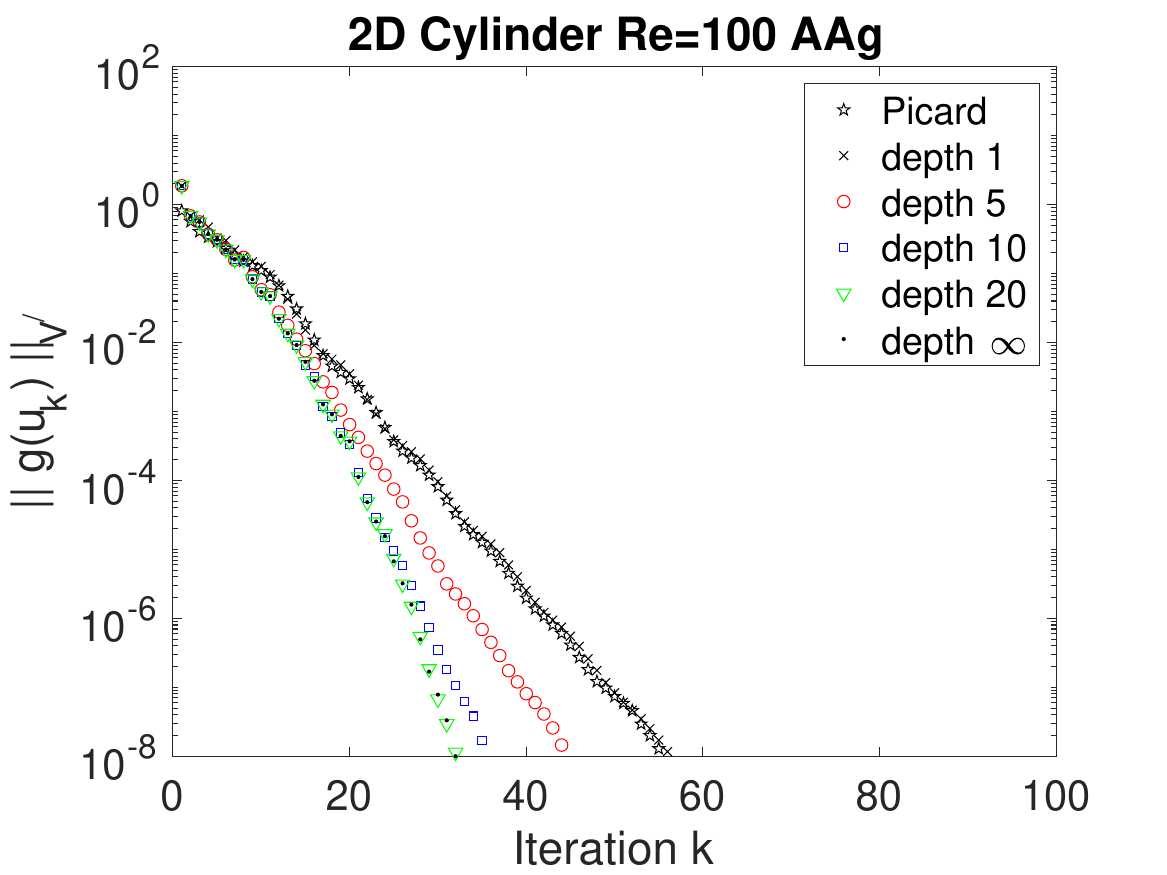}  
\includegraphics[width = .32\textwidth, height=.23\textwidth,viewport=0 0 530 420, clip]{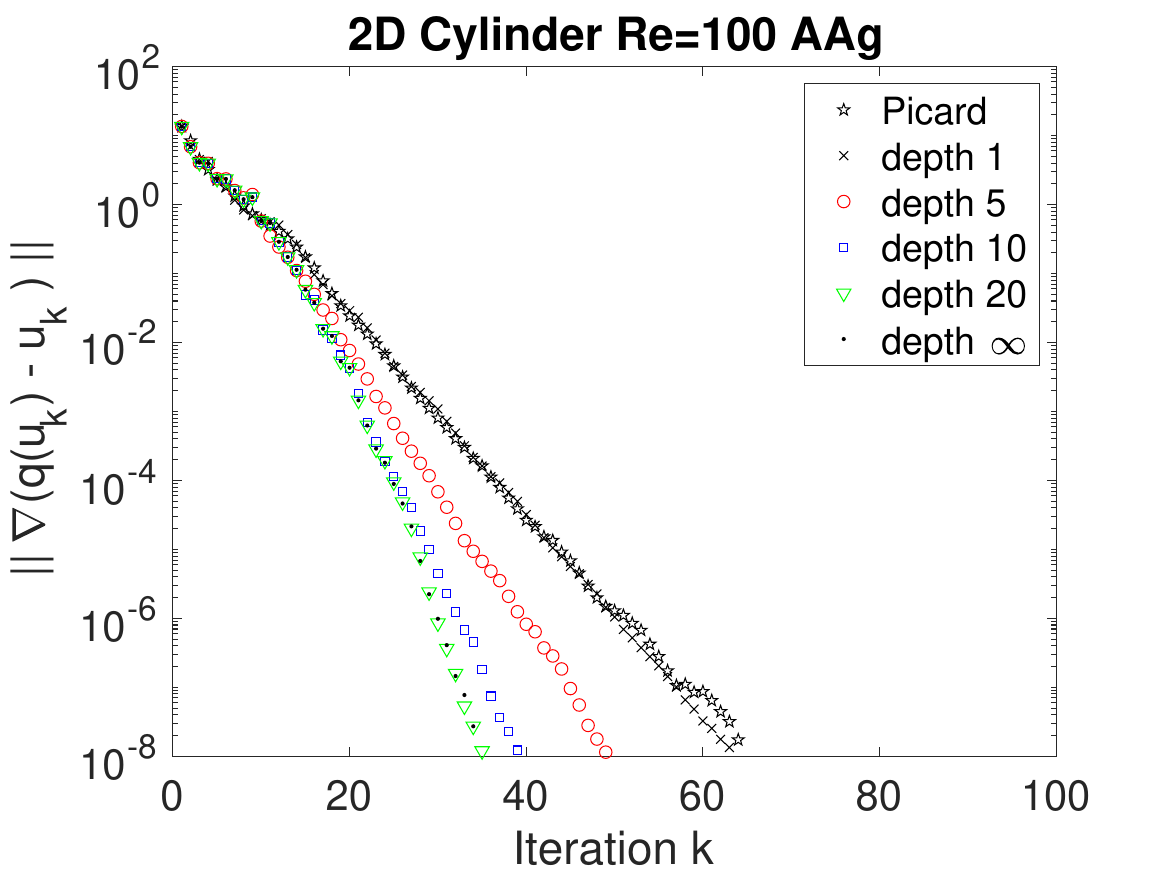}  
\includegraphics[width = .32\textwidth, height=.23\textwidth,viewport=0 0 530 420, clip]{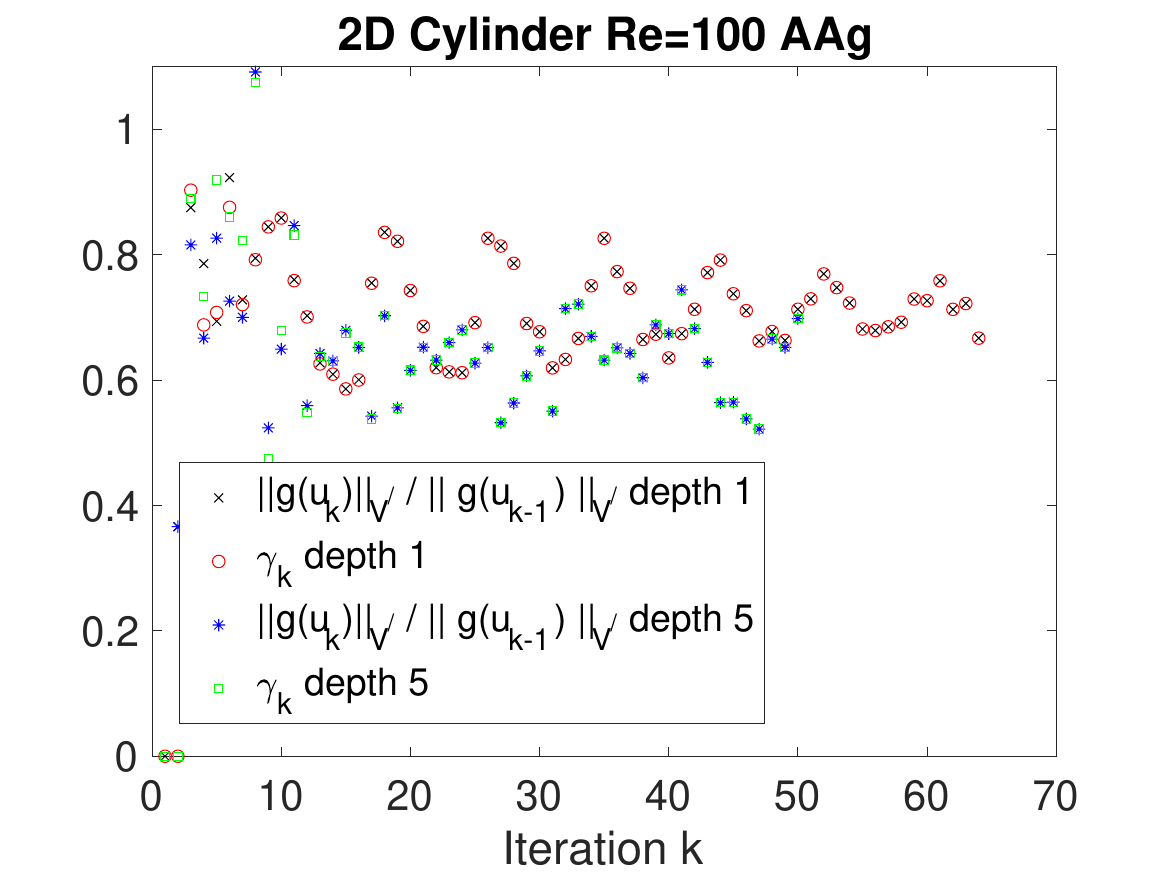} \\ 
\includegraphics[width = .32\textwidth, height=.23\textwidth,viewport=0 0 530 420, clip]{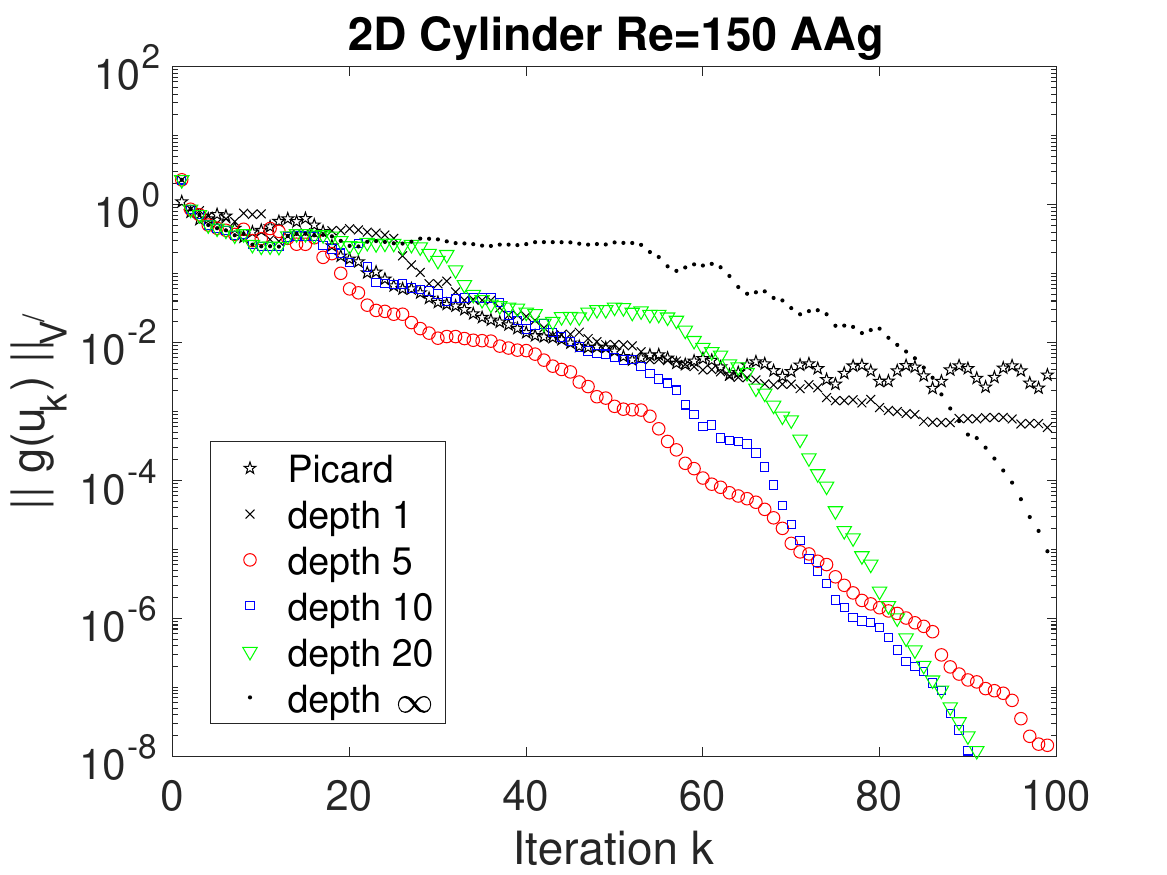}  
\includegraphics[width = .32\textwidth, height=.23\textwidth,viewport=0 0 530 420, clip]{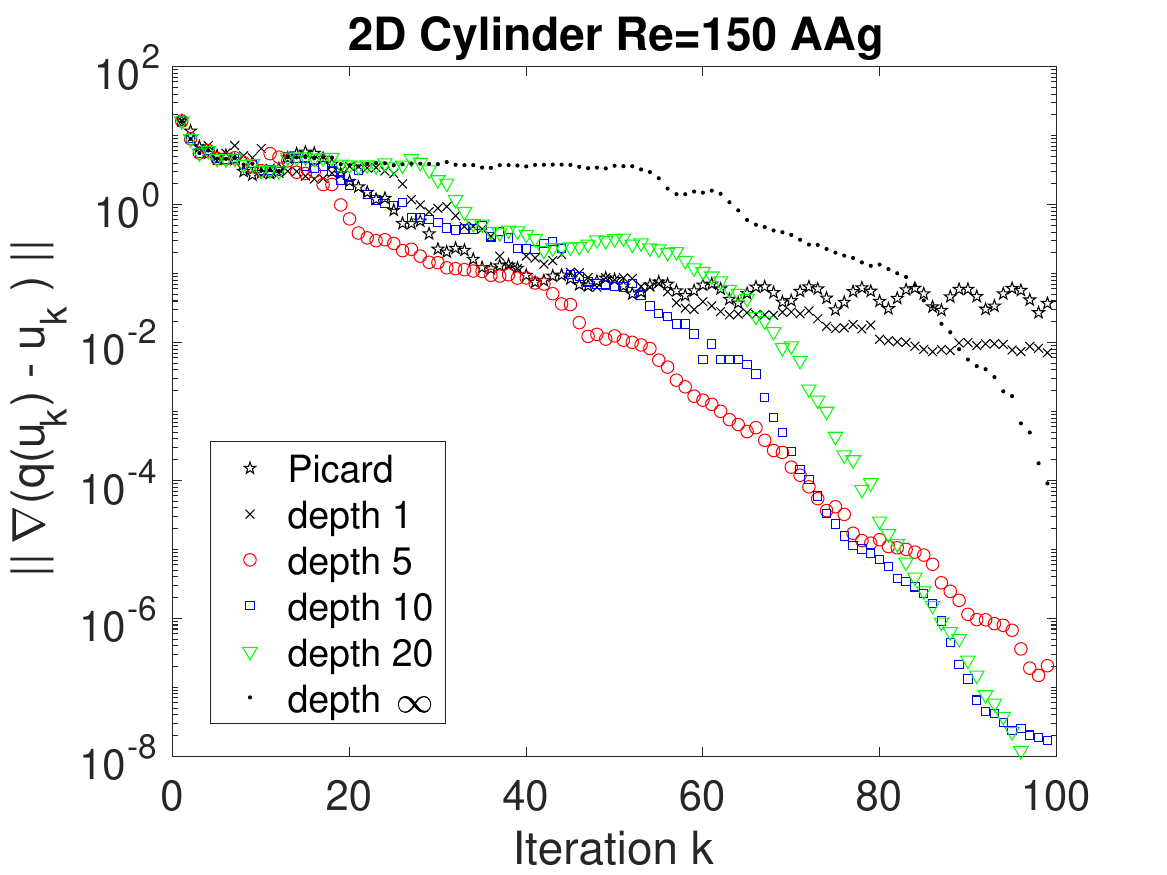}  
\includegraphics[width = .32\textwidth, height=.23\textwidth,viewport=0 0 530 420, clip]{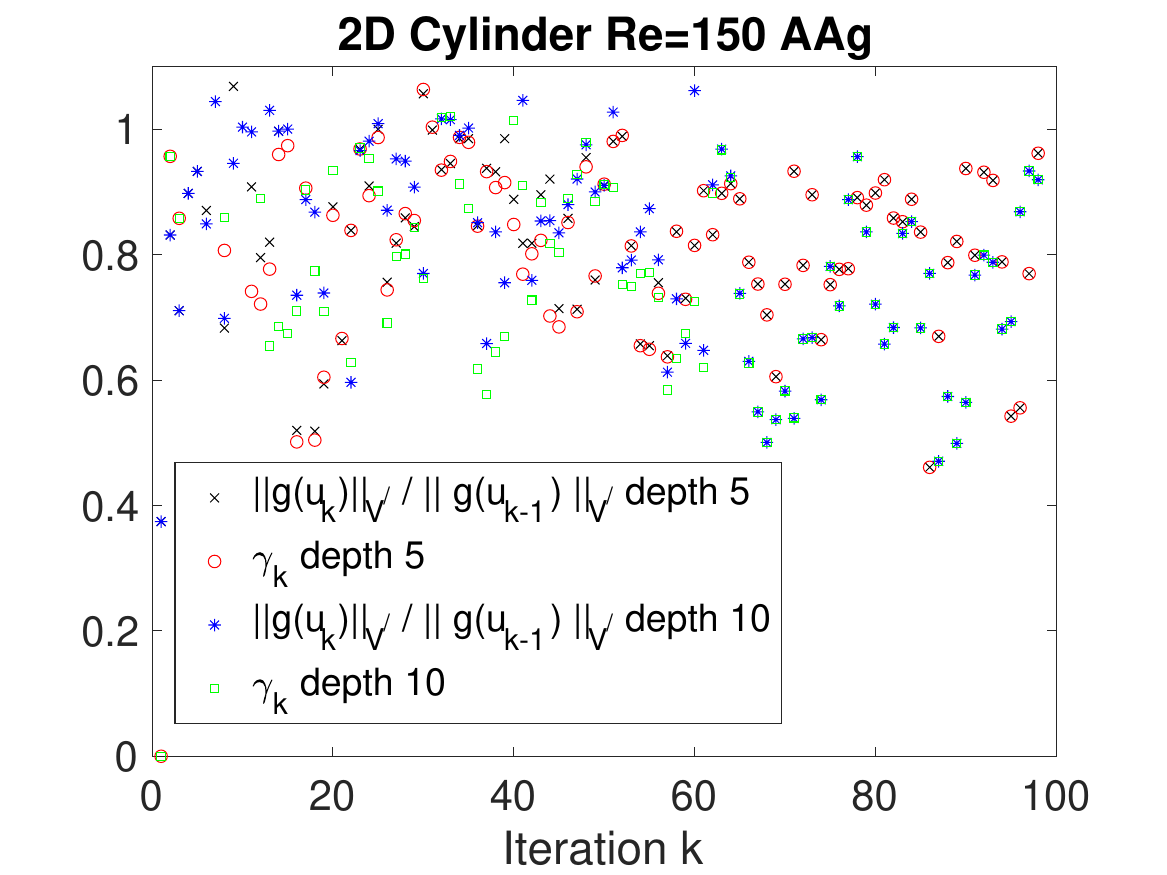} \\
\includegraphics[width = .32\textwidth, height=.23\textwidth,viewport=0 0 530 420, clip]{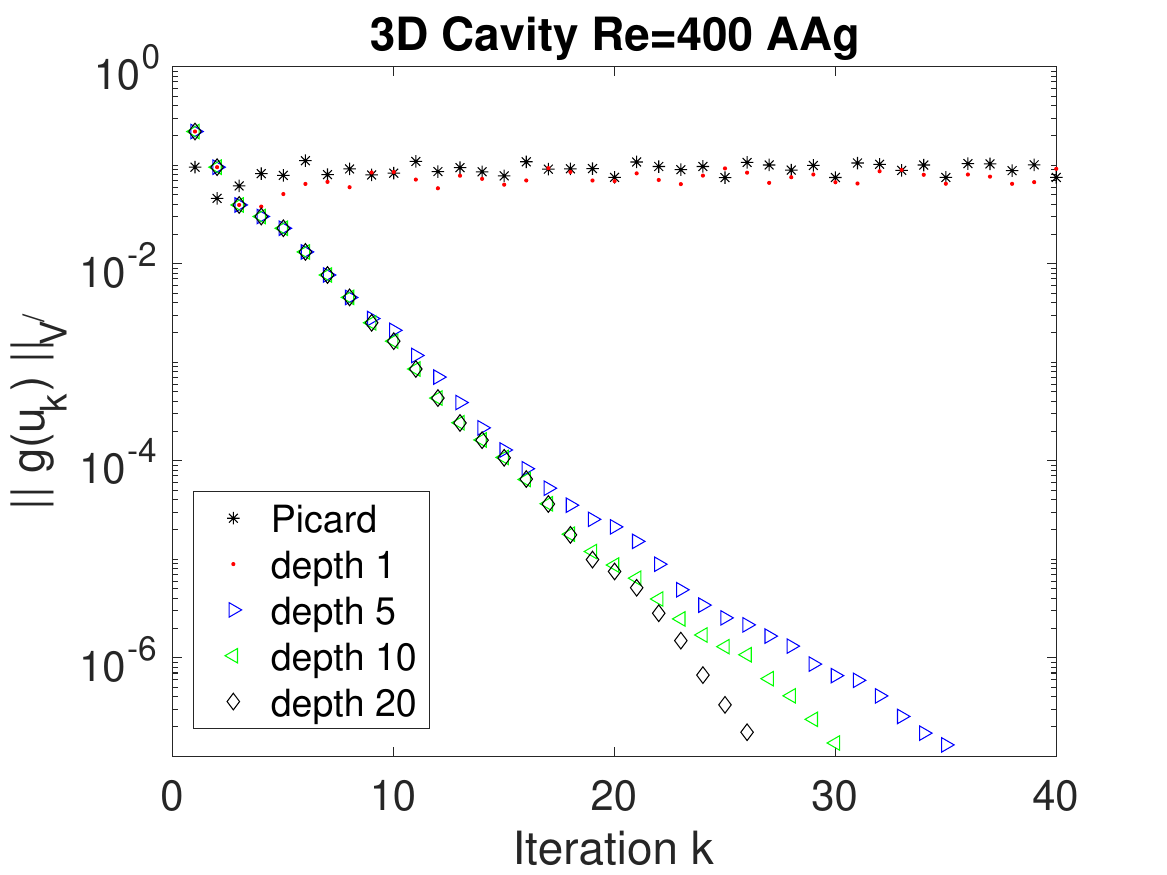}  
\includegraphics[width = .32\textwidth, height=.23\textwidth,viewport=0 0 530 420, clip]{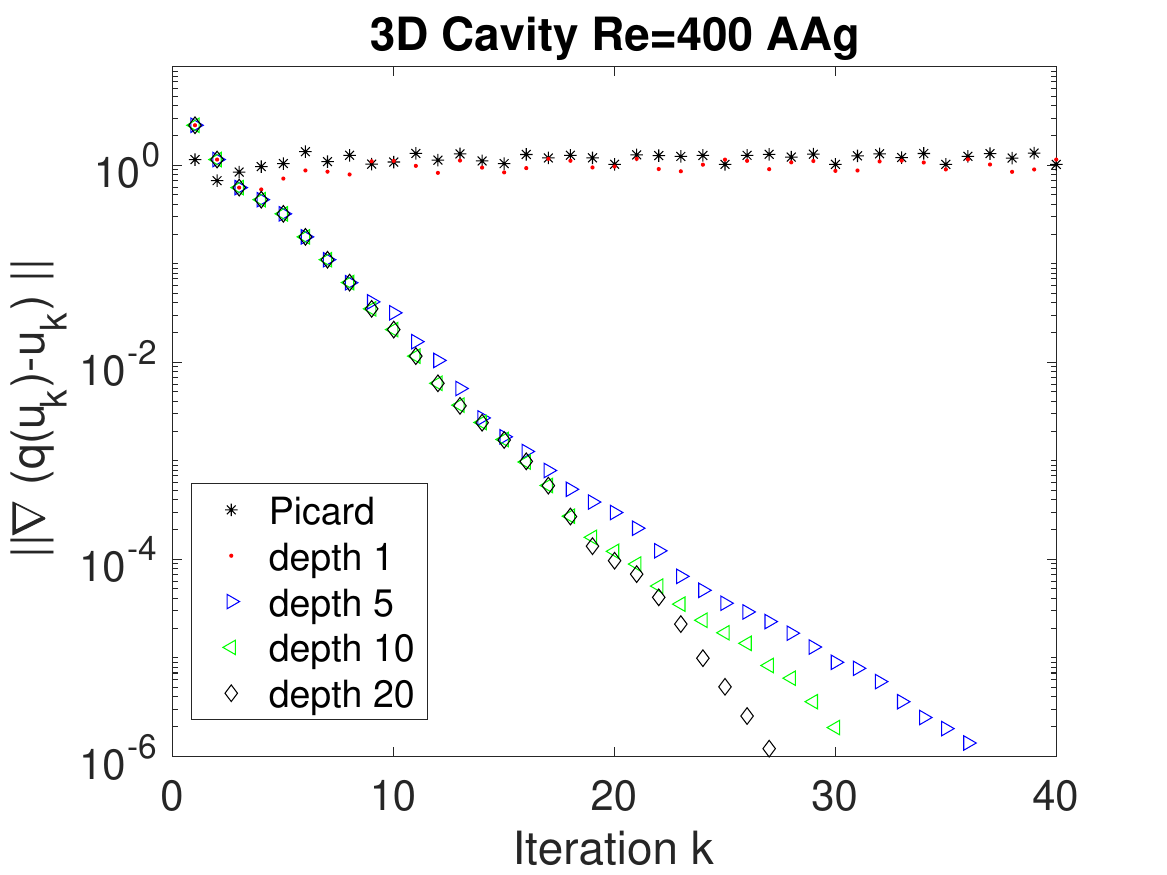}  
\includegraphics[width = .32\textwidth, height=.23\textwidth,viewport=0 0 530 420, clip]{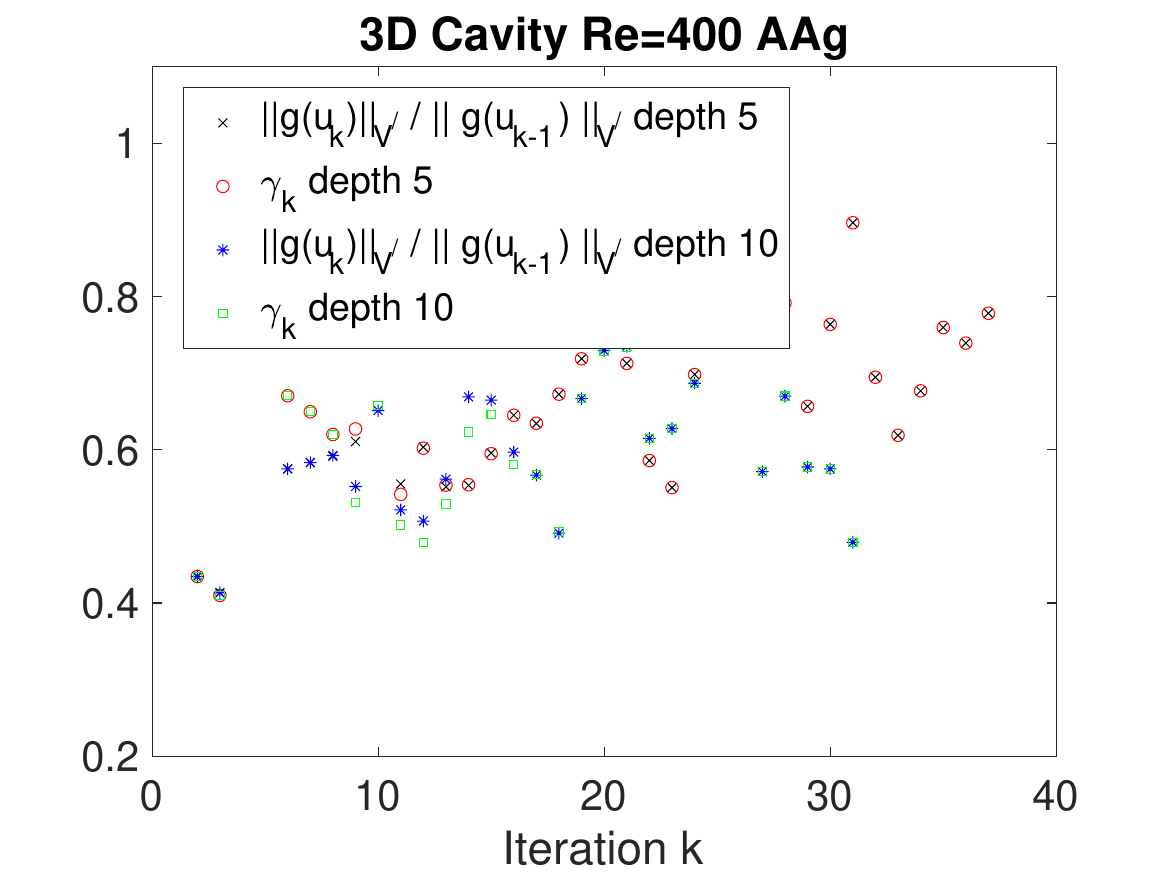} \\
\includegraphics[width = .32\textwidth, height=.23\textwidth,viewport=0 0 530 420, clip]{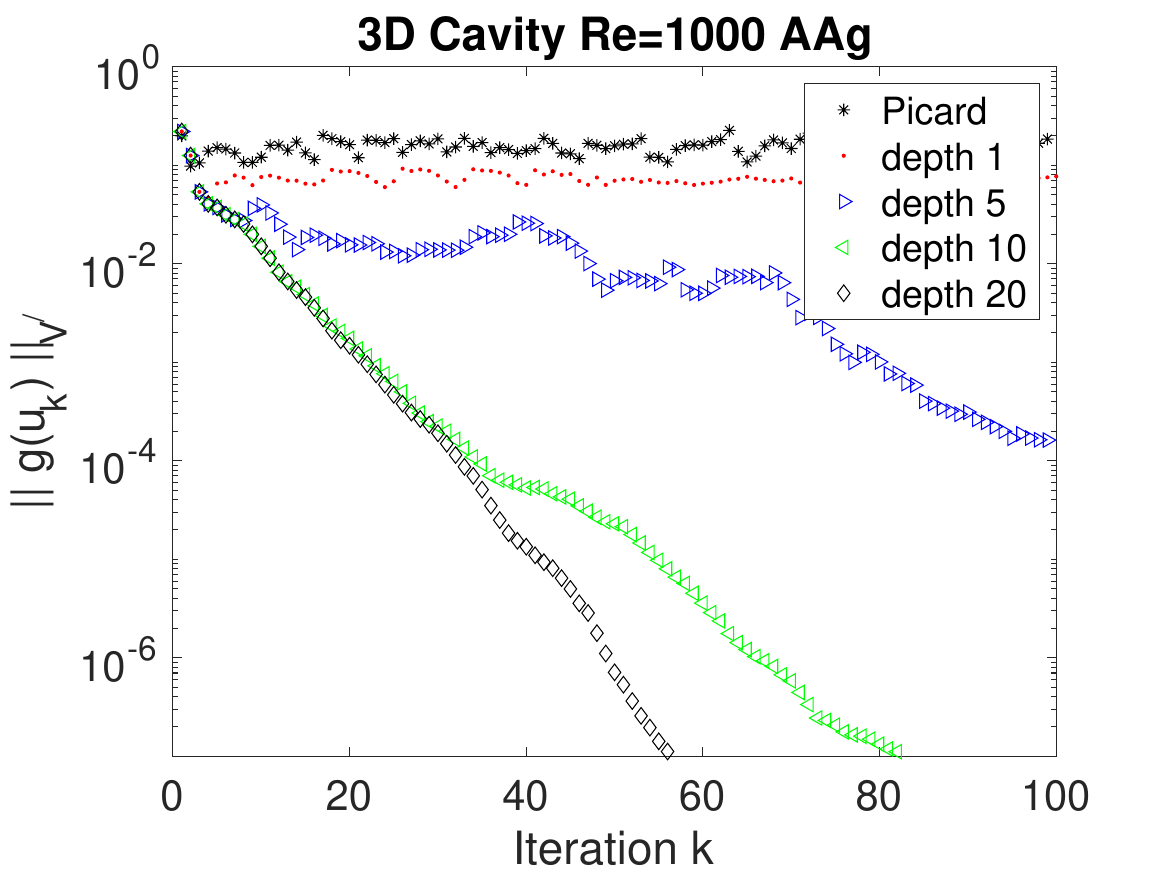}  
\includegraphics[width = .32\textwidth, height=.23\textwidth,viewport=0 0 530 420, clip]{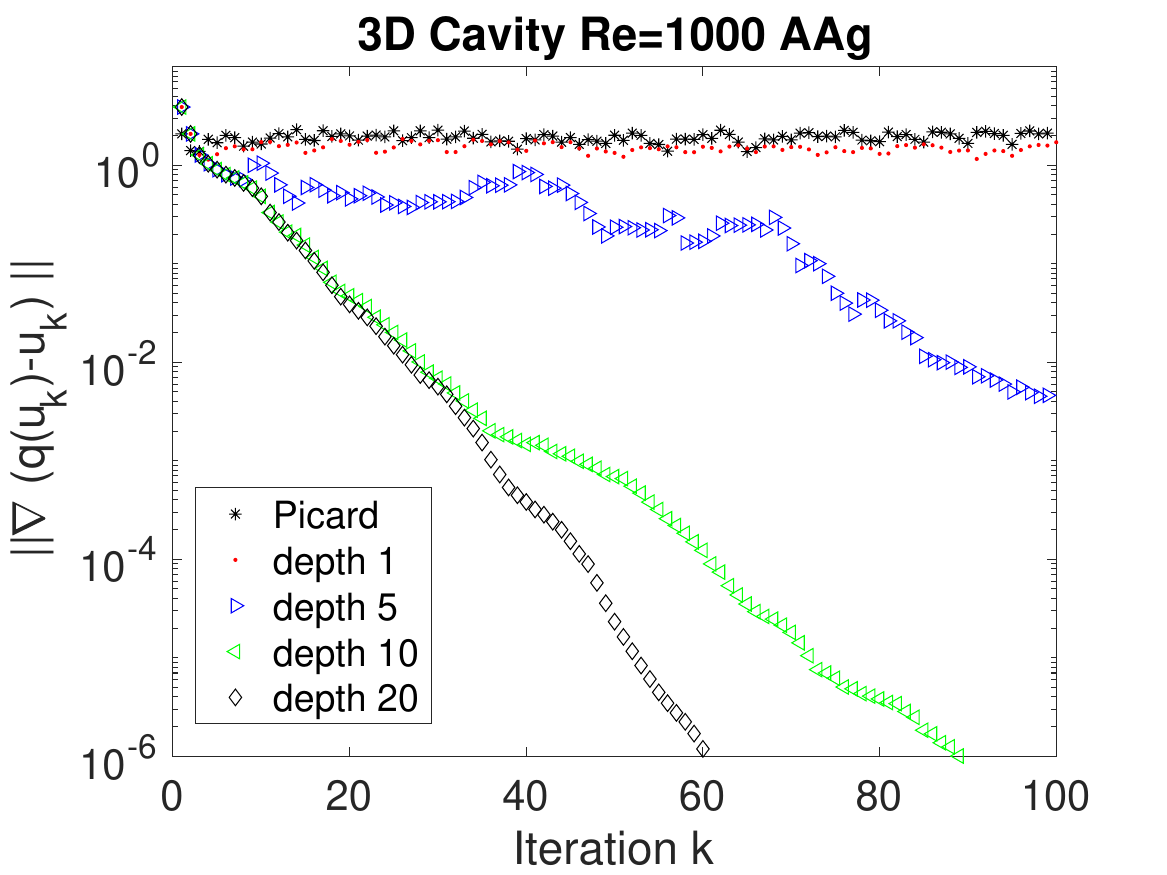}  
\includegraphics[width = .32\textwidth, height=.23\textwidth,viewport=0 0 530 420, clip]{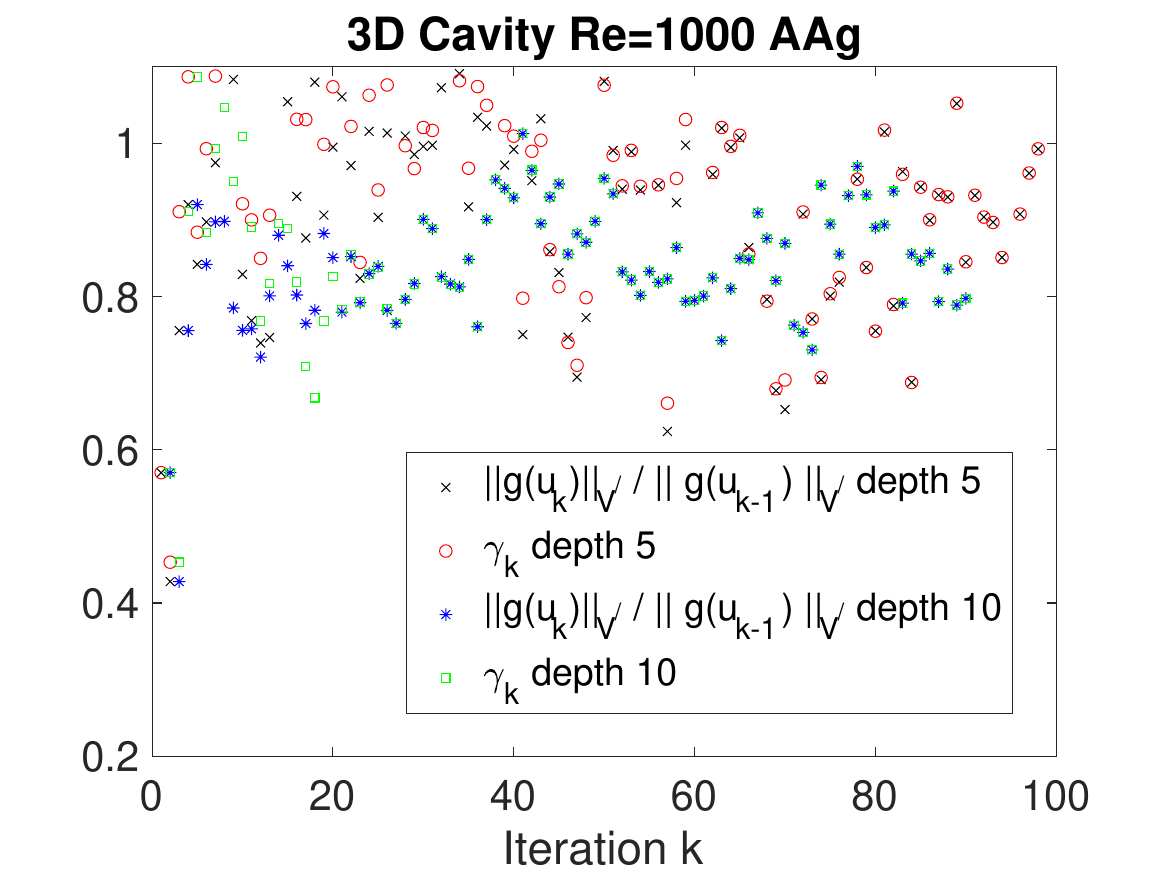} \\ 
\includegraphics[width = .32\textwidth, height=.23\textwidth,viewport=0 0 530 420, clip]{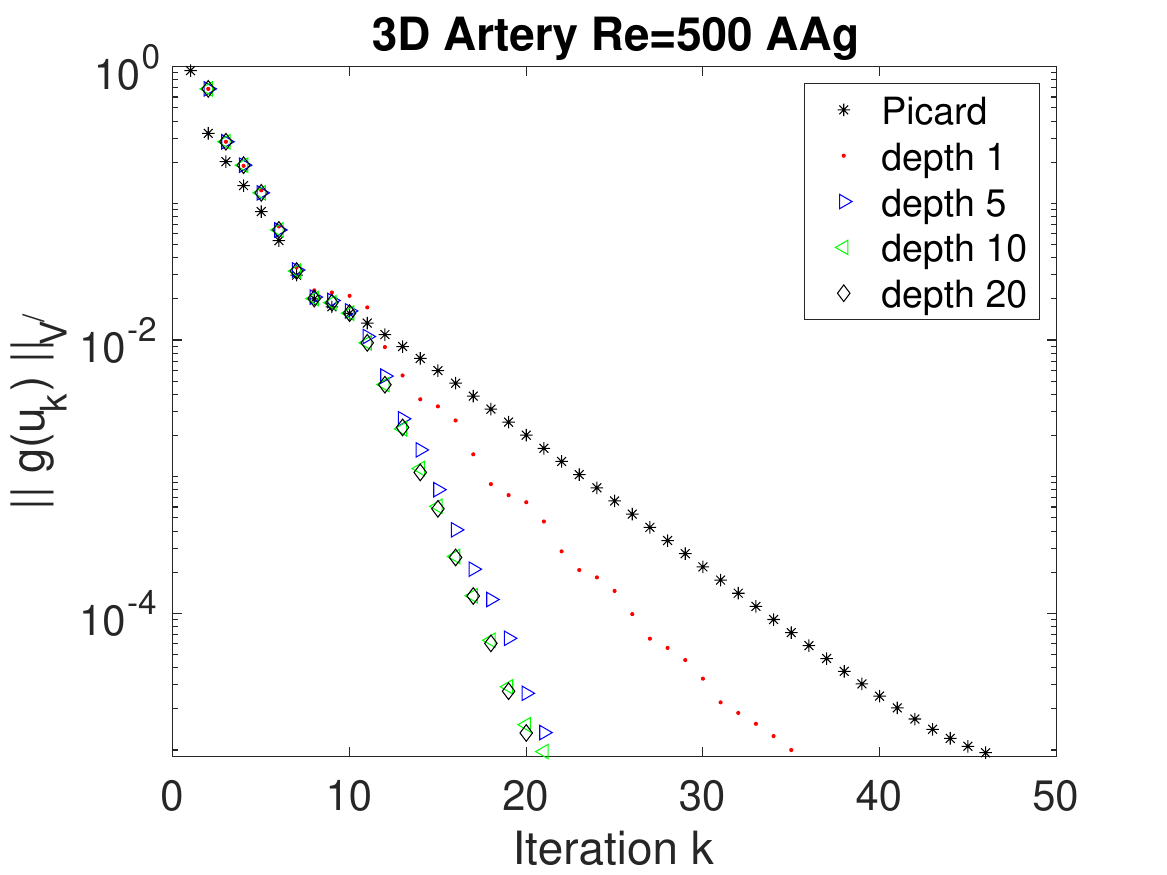}  
\includegraphics[width = .32\textwidth, height=.23\textwidth,viewport=0 0 530 420, clip]{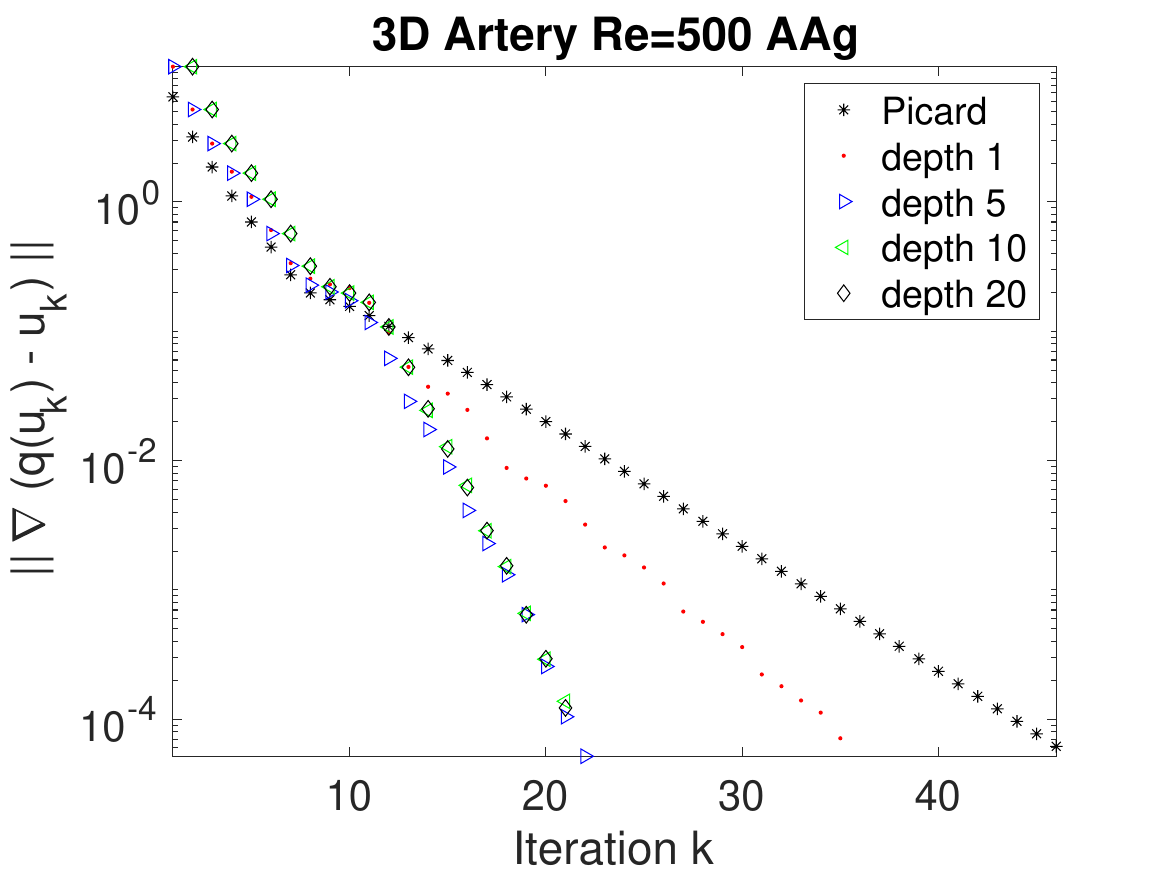}  
\includegraphics[width = .32\textwidth, height=.23\textwidth,viewport=0 0 530 420, clip]{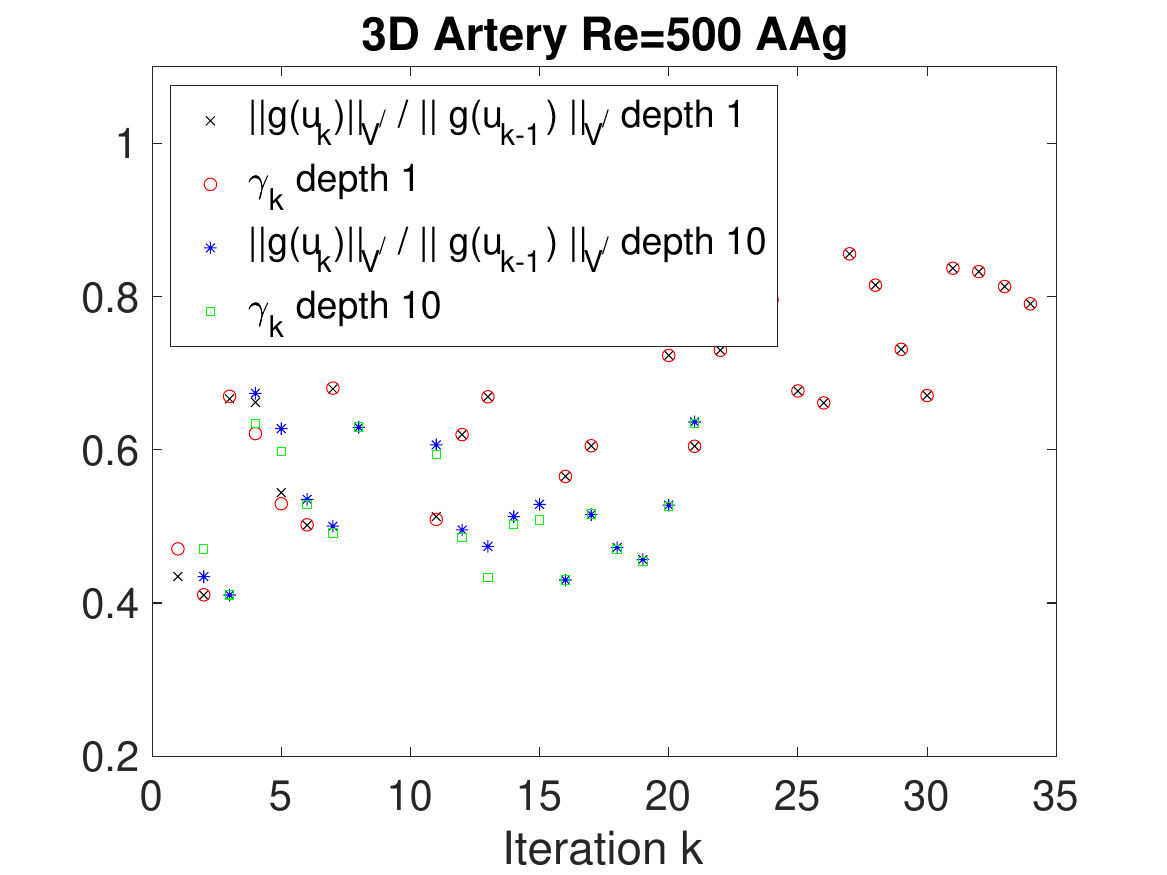} 
\caption{\label{dcconvopt} The plots above show AAg convergence in $\| g(u_k)\|_{V'}$, $\| \nabla (q(u_k)-u_k)\|$, and $\gamma_k$ with $\frac{\| g(u_k)\|_{V'}}{\| g(u_{k-1})\|_{V'}}$ (left to right), for 
2D channel flow past a block, 3D driven cavity, and 3D stenotic artery.}
\end{figure}

We now consider convergence behavior of AAg-Picard on the 2D channel flow past a block with $Re$=100 and 150, the 3D driven cavity with $Re$=400 and 1000, and the 3D stenotic artery model.  We compute convergence as the nonlinear residual in the $V'$ norm $\| g(u_k) \|_{V'}$ (to match theory), and the Picard residual in the $H^1_0$ norm $\| \nabla (q(u_k)-u_k)\|$ (to match the usual Picard convergence theory \cite{GR86,temam} and AA-Picard theory \cite{PRX19,PR25}).  We observe very similar convergence behavior from these norms in all tests, although the actual values of the two norms are typically off by a constant. Convergence of these residuals measured in $\ell^2$ norms also show similar convergence behavior (plots omitted).  We also compute $\gamma_k$ and plot it along with the linear convergence ratio $\frac{ \| g(u_k)\|_{V'}}{ \| g(u_{k-1})\|_{V'}}$, to test the accuracy of $\gamma_k$ as the linear convergence rate predictor.

Results are shown in Figure \ref{dcconvopt}, for AAg-Picard with varying depths (depth $\infty$ means $m=k$ at iteration $k$ for AAg), along with usual Picard.  First we note that the first and second columns look very similar, which implies that convergence behavior in the $V'$ and $H^1_0$ norms are essentially the same (in later tests we will only use $V'$).   Overall, we see similar behavior in all tests.  In all cases, AAg-Picard shows significant improvement over Picard, with AAg either dramatically accelerating convergence and even enabling it since Picard fails for $Re$=150  2D channel flow past a block and both 3D cavity problems.  We also observe that, except for the $Re$=150  2D channel flow past a block tests, convergence is improved as the depth is increased.  

In the third column of Figure \ref{dcconvopt}, we observe in all cases that for larger enough $k$, $\gamma_k$ agrees (essentially) exactly with the observed linear convergence ratio at each step.  This suggests that our analysis is sharp: $\gamma_k$ is a very good predictor of the linear convergence rate, once the higher order terms become negligible.  Hence it fits that for the easier problems the agreement of $\frac{ \| g(u_k)\|_{V'}}{ \| g(u_{k-1})\|_{V'}}$ and $\gamma_k$ is reached for smaller $k$, while for the harder problems they do not agree until $k$ is much larger.  We note that for classical AA, there is no known way to predict the linear convergence rate this accurately.

\subsection{Convergence comparison for AAg, AA and NGMRES}

\begin{figure}[h!]
\center
\includegraphics[width = .32\textwidth, height=.23\textwidth,viewport=0 0 530 420, clip]{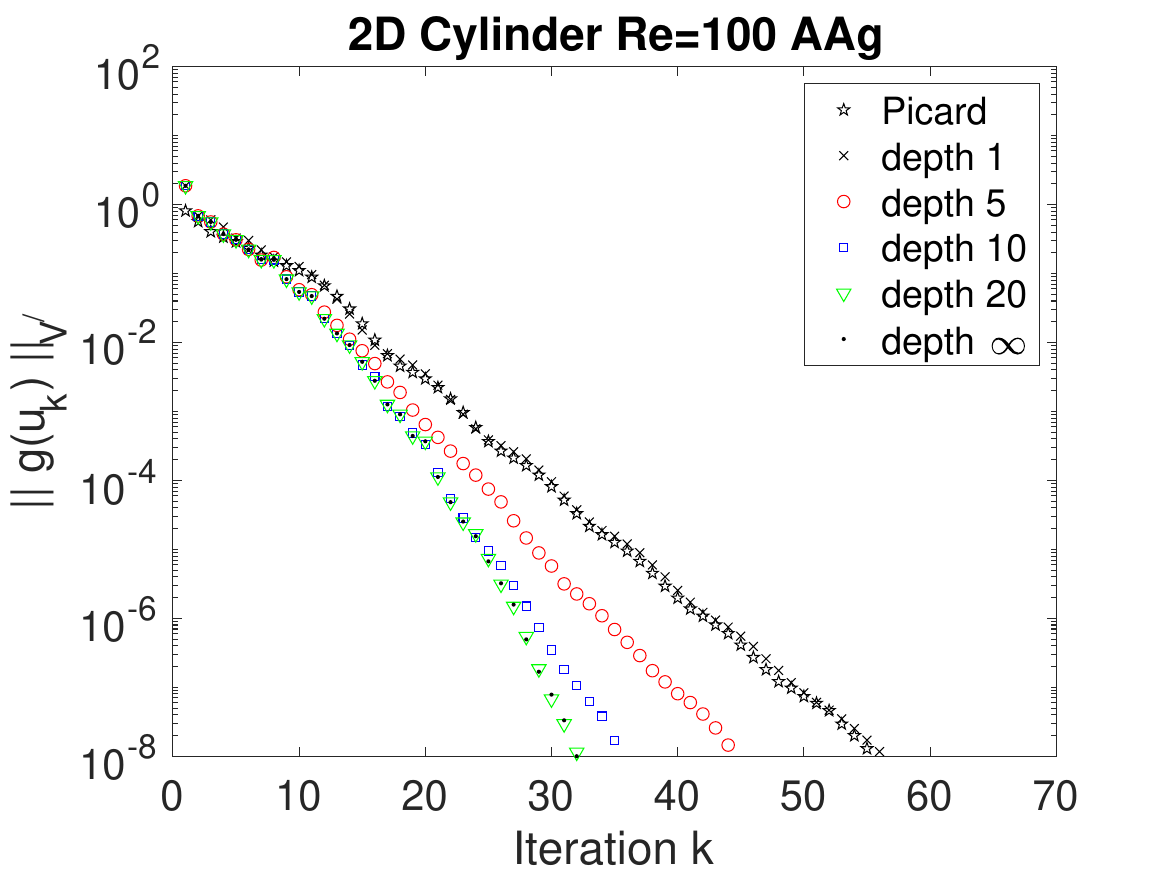}  
\includegraphics[width = .32\textwidth, height=.23\textwidth,viewport=0 0 530 420, clip]{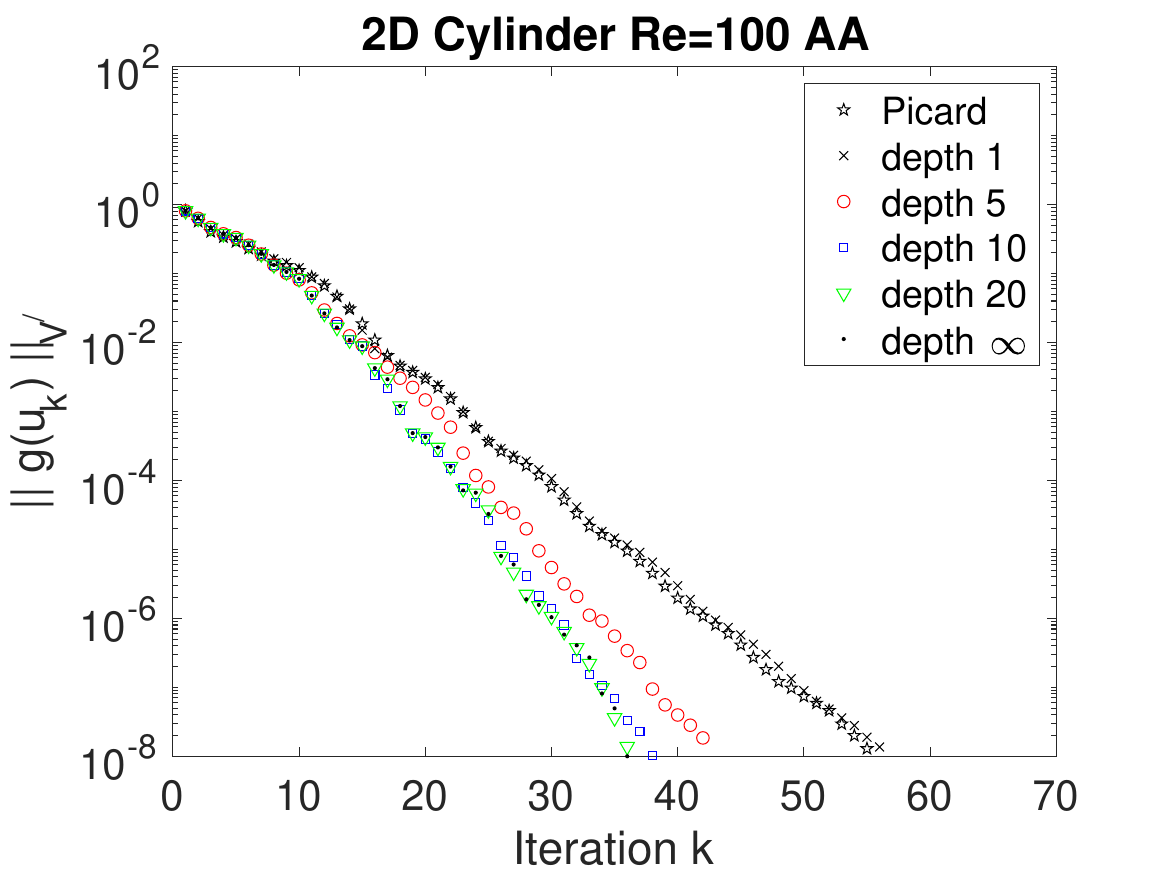}  
\includegraphics[width = .32\textwidth, height=.23\textwidth,viewport=0 0 530 420, clip]{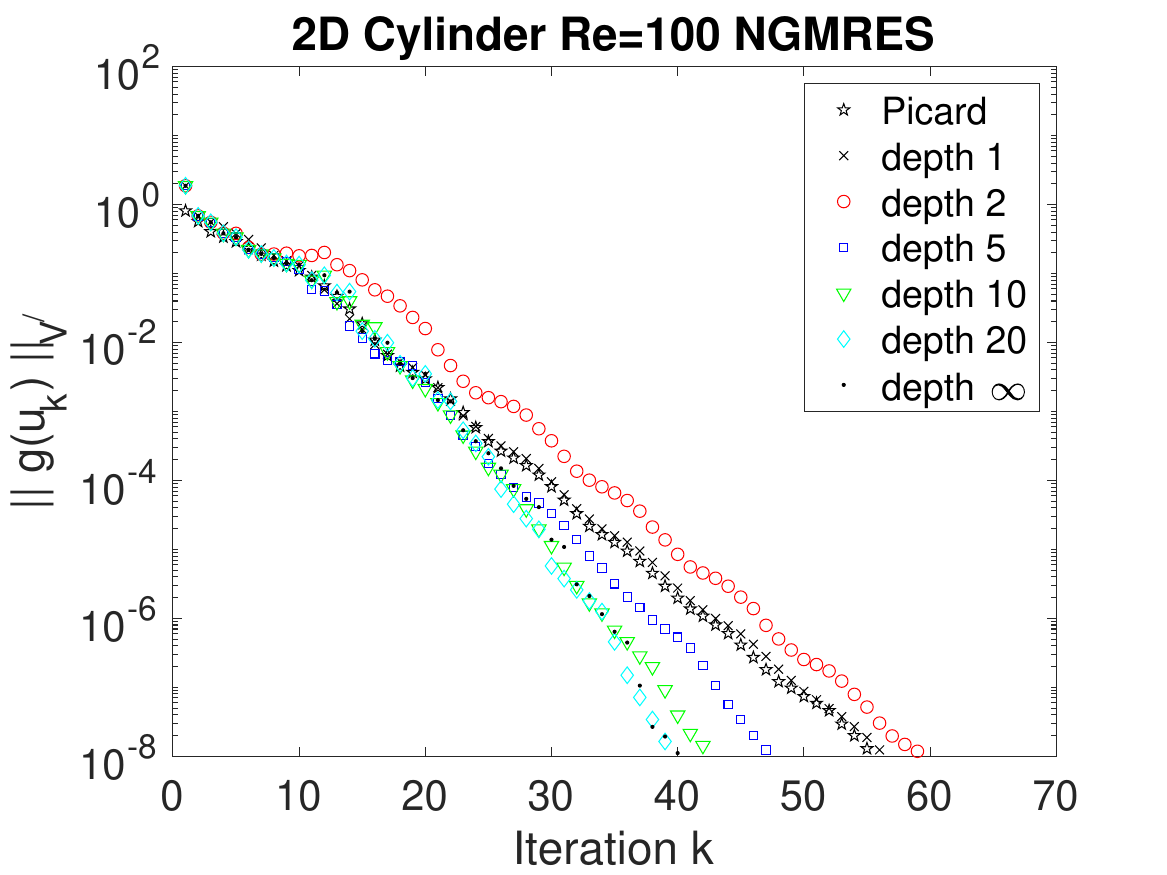} \\
\includegraphics[width = .32\textwidth, height=.23\textwidth,viewport=0 0 530 420, clip]{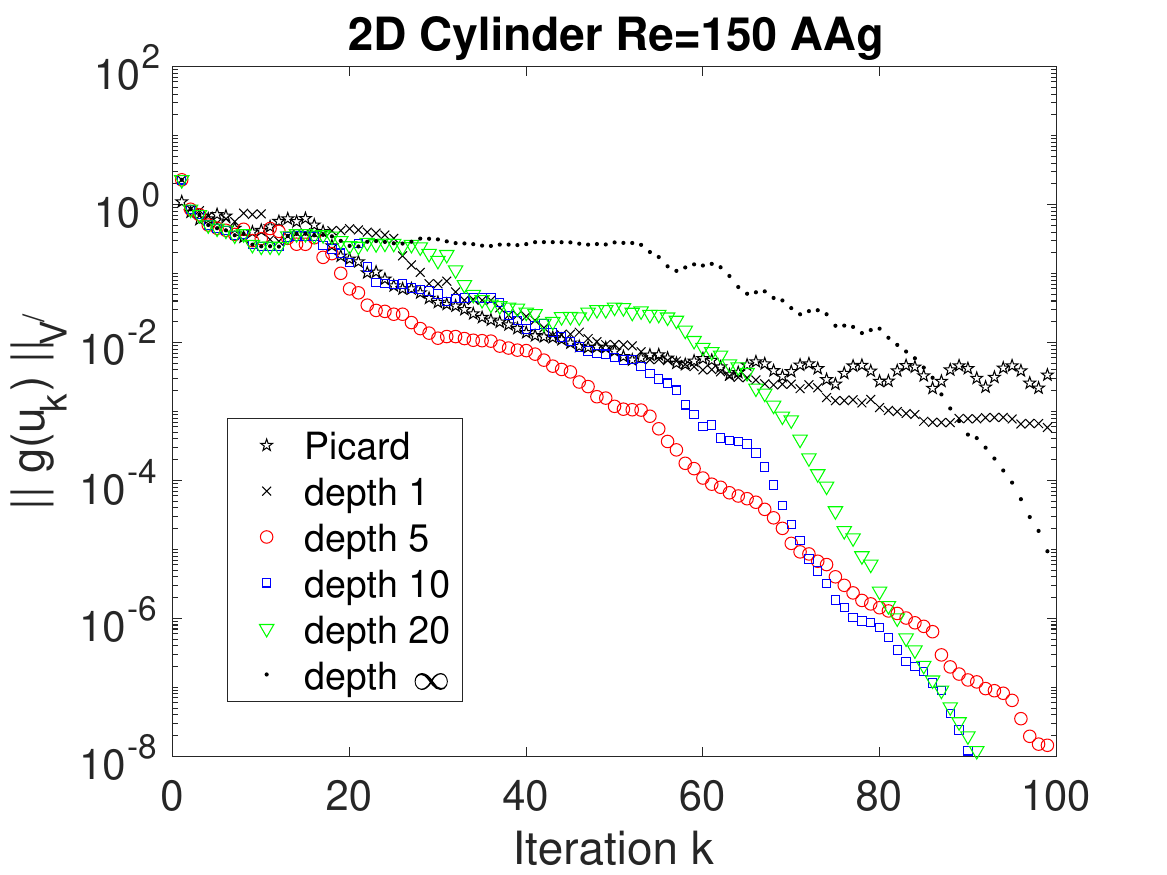}  
\includegraphics[width = .32\textwidth, height=.23\textwidth,viewport=0 0 530 420, clip]{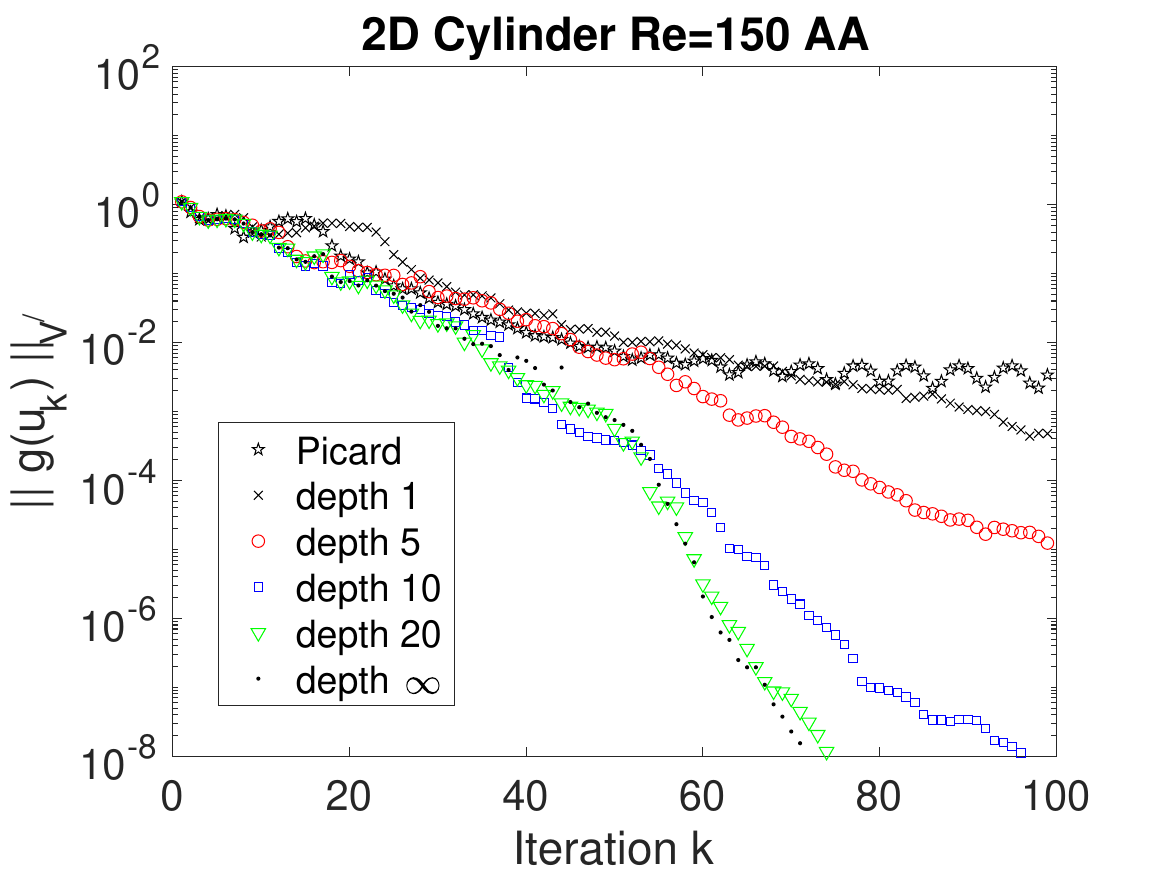}  
\includegraphics[width = .32\textwidth, height=.23\textwidth,viewport=0 0 530 420, clip]{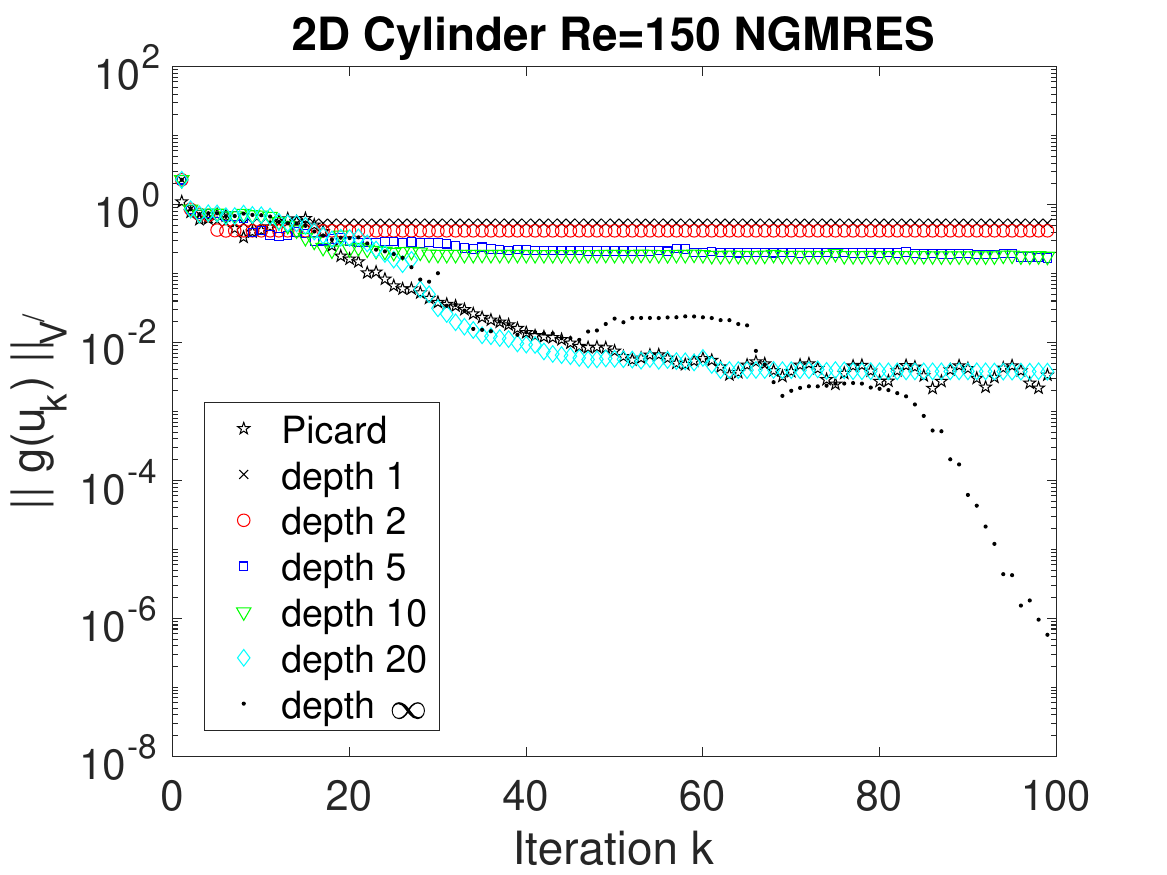} \\
\includegraphics[width = .32\textwidth, height=.23\textwidth,viewport=0 0 530 420, clip]{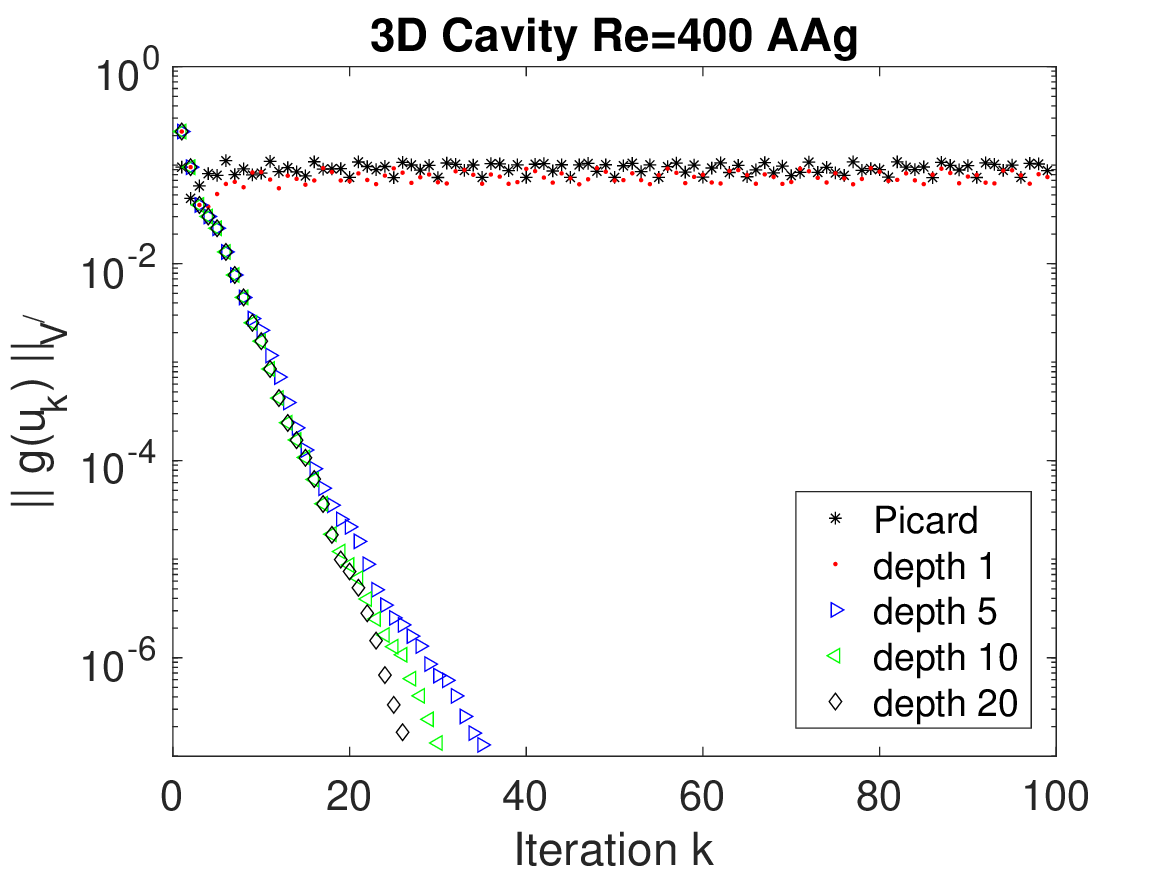}  
\includegraphics[width = .32\textwidth, height=.23\textwidth,viewport=0 0 530 420, clip]{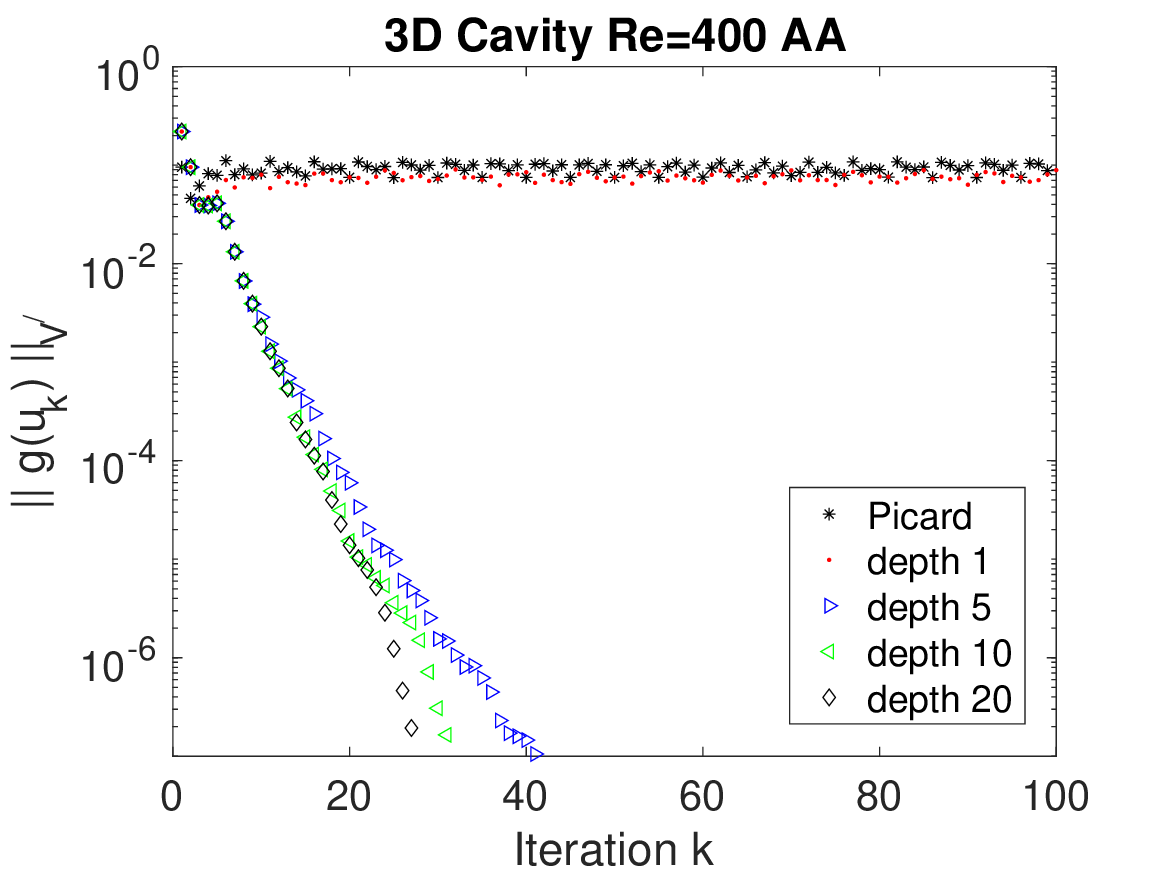}  
\includegraphics[width = .32\textwidth, height=.23\textwidth,viewport=0 0 530 420, clip]{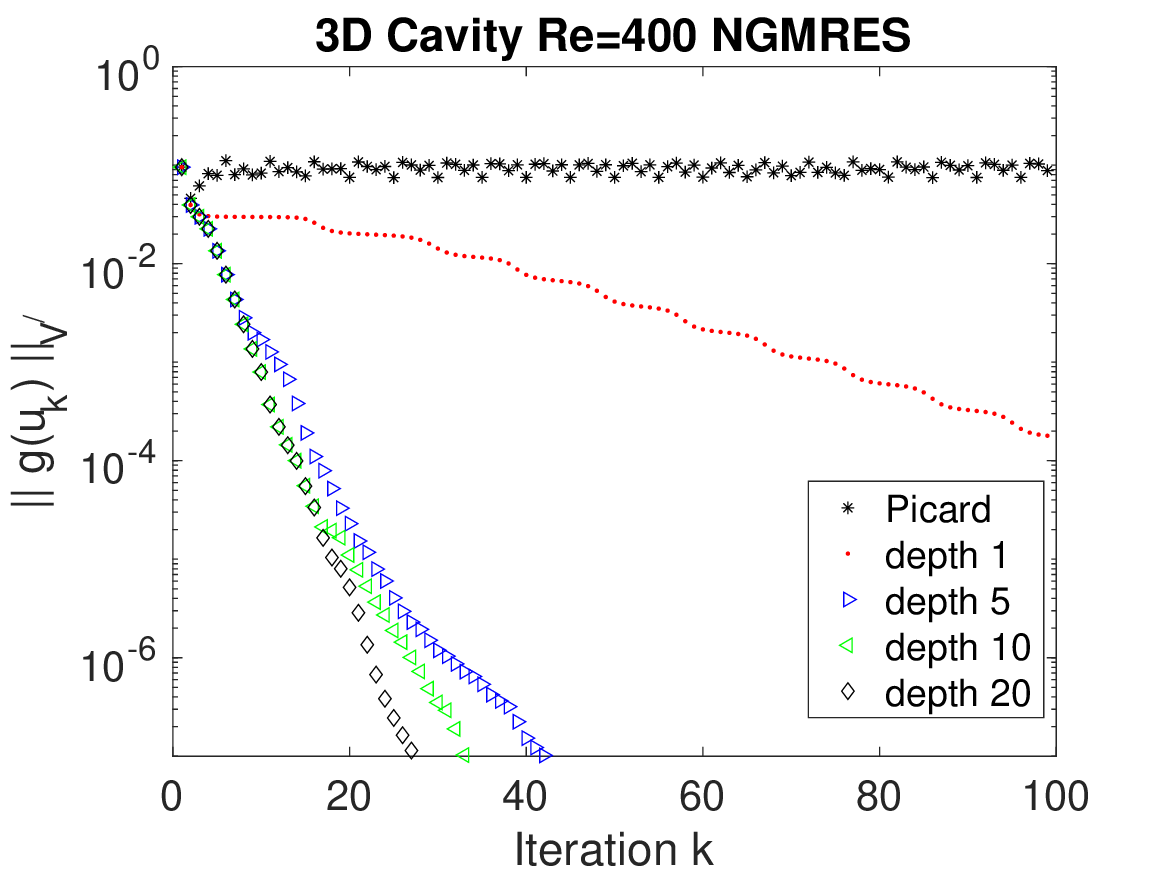} \\
\includegraphics[width = .32\textwidth, height=.23\textwidth,viewport=0 0 530 420, clip]{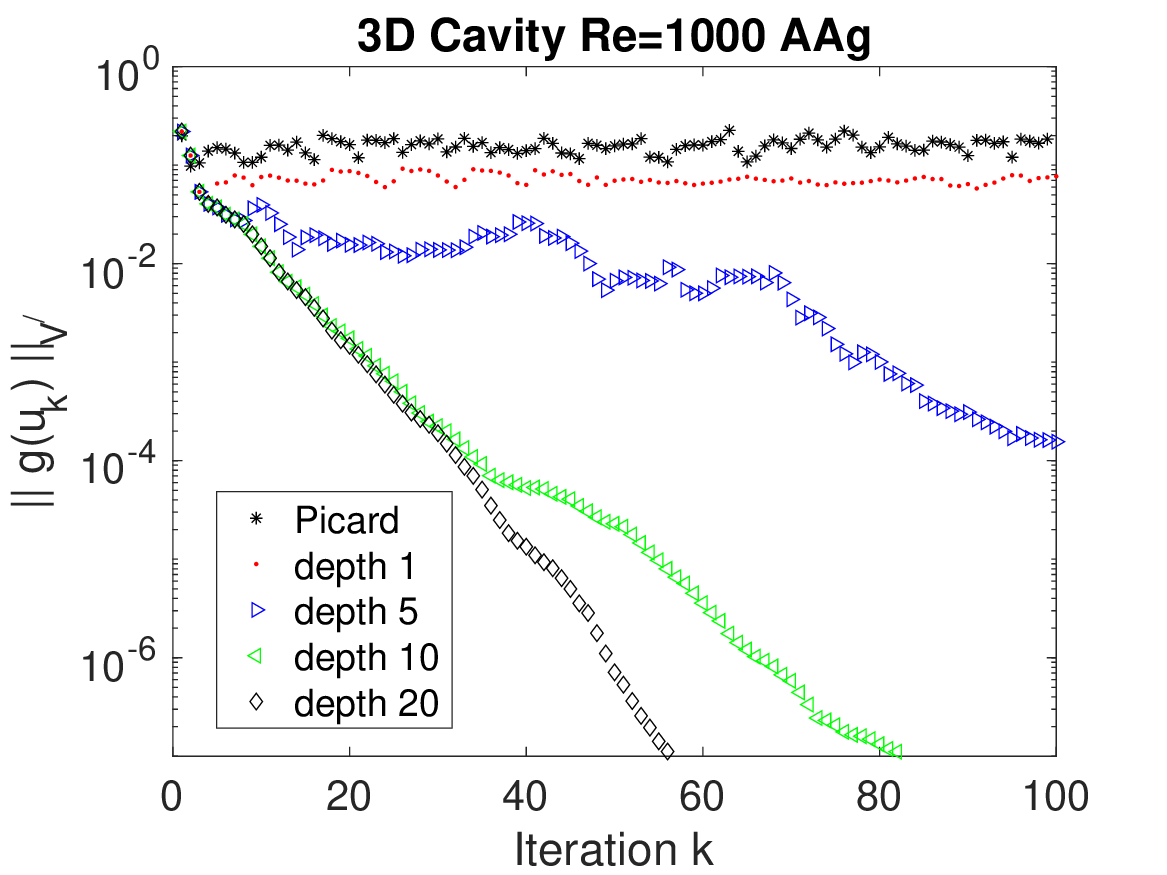}  
\includegraphics[width = .32\textwidth, height=.23\textwidth,viewport=0 0 530 420, clip]{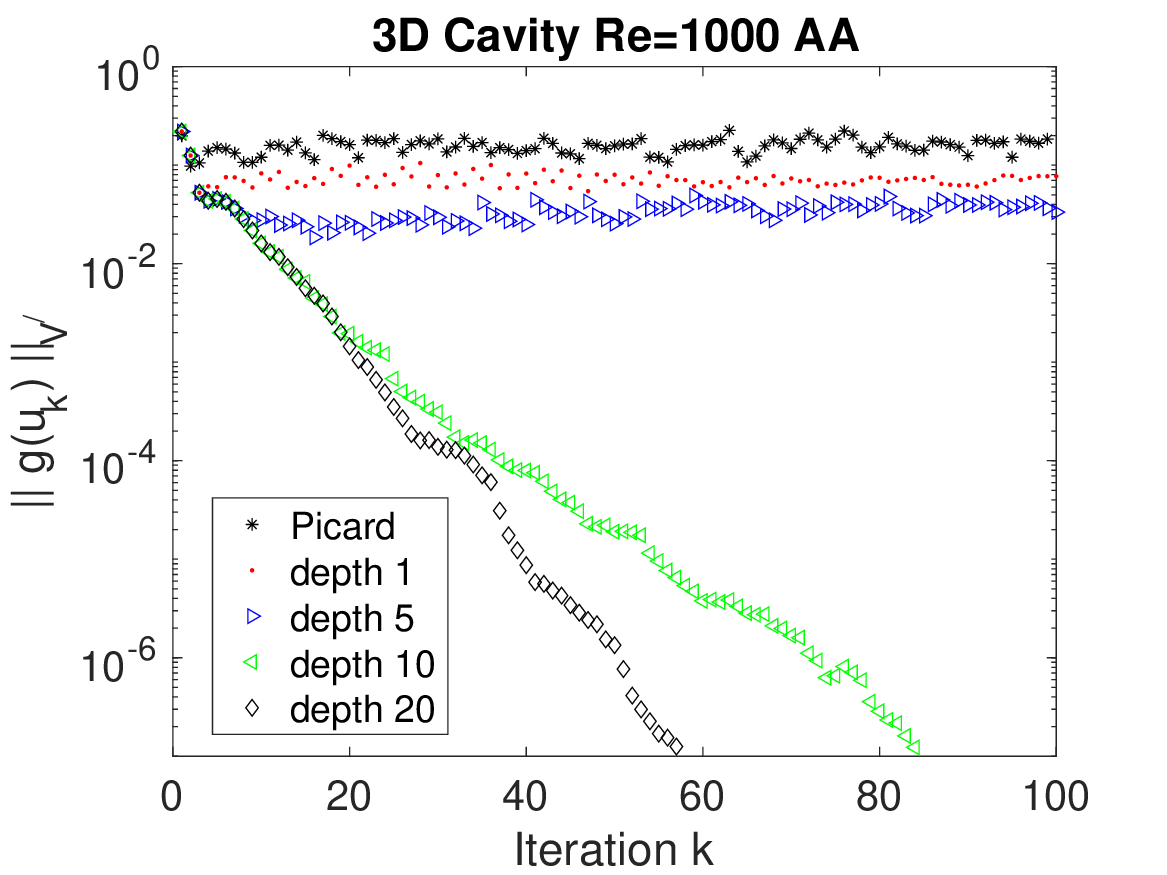}  
\includegraphics[width = .32\textwidth, height=.23\textwidth,viewport=0 0 530 420, clip]{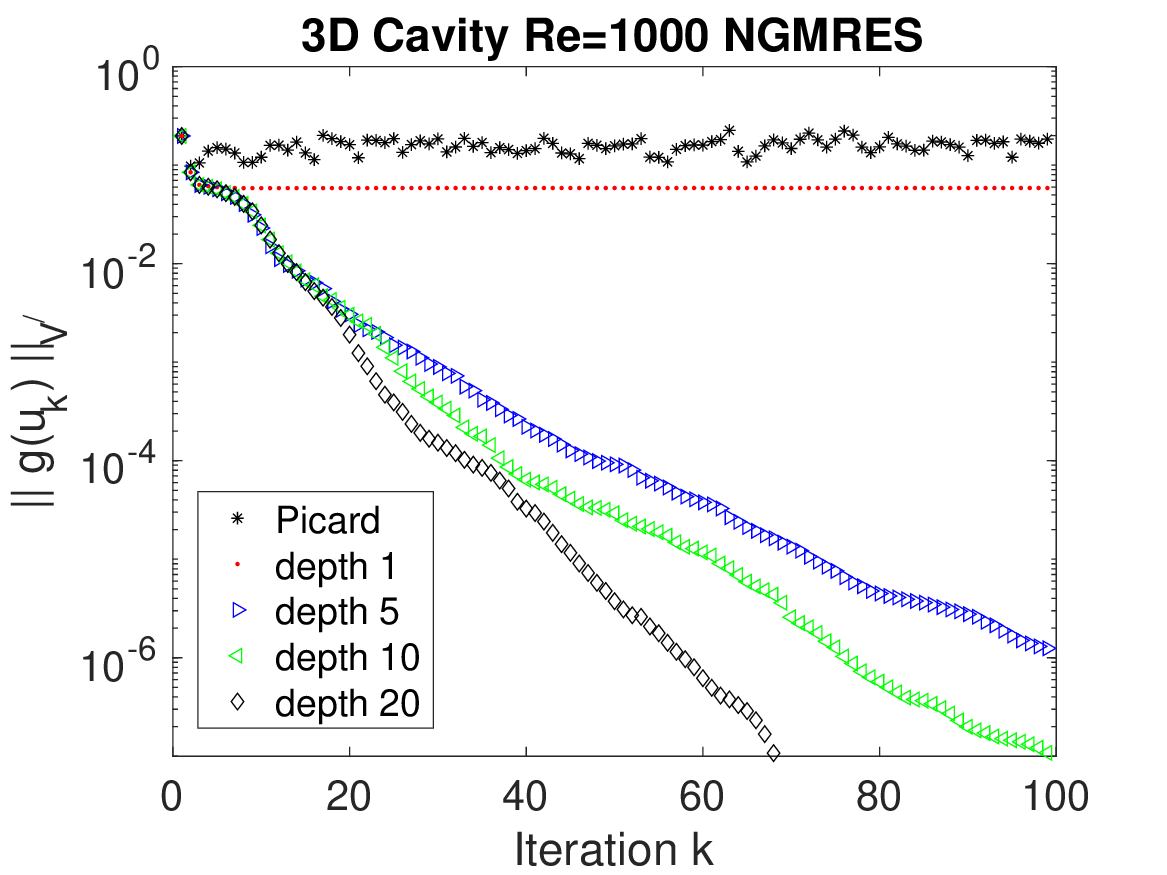} \\\includegraphics[width = .32\textwidth, height=.23\textwidth,viewport=0 0 530 420, clip]{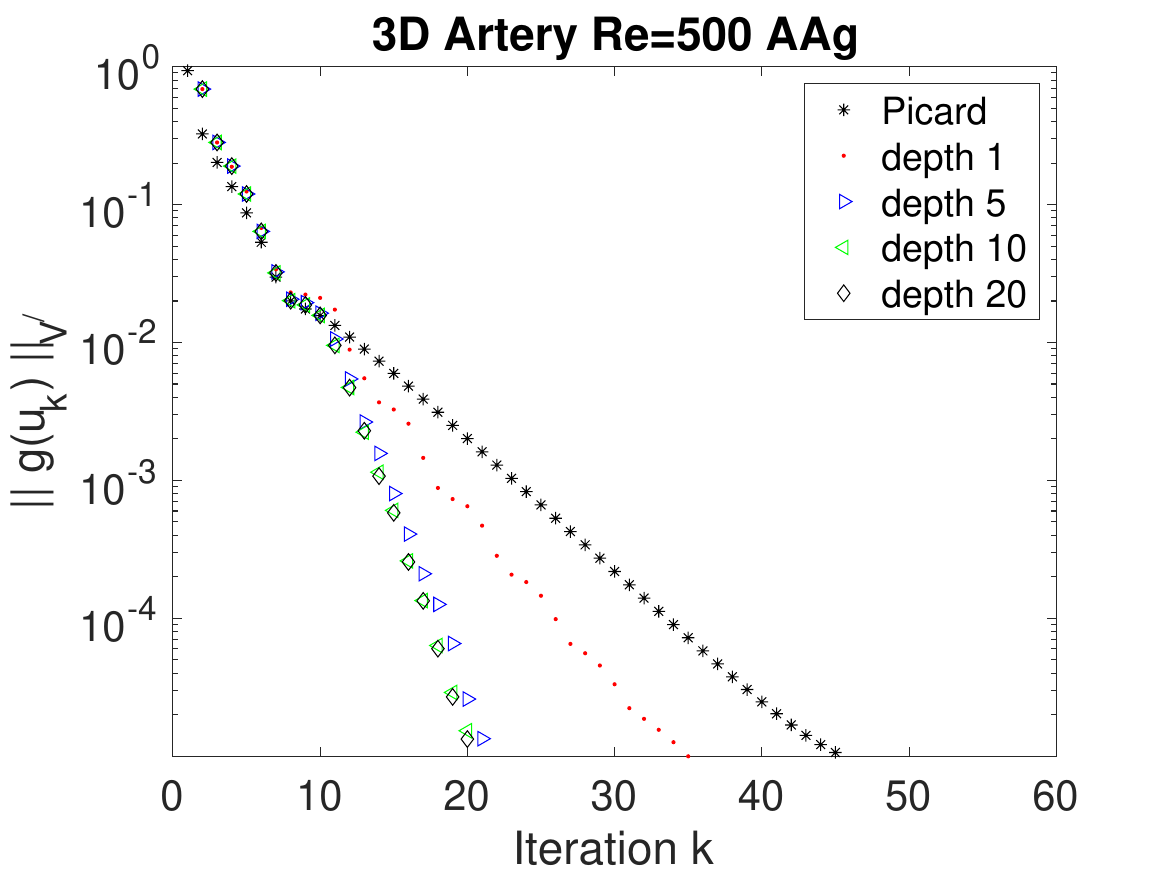}  
\includegraphics[width = .32\textwidth, height=.23\textwidth,viewport=0 0 530 420, clip]{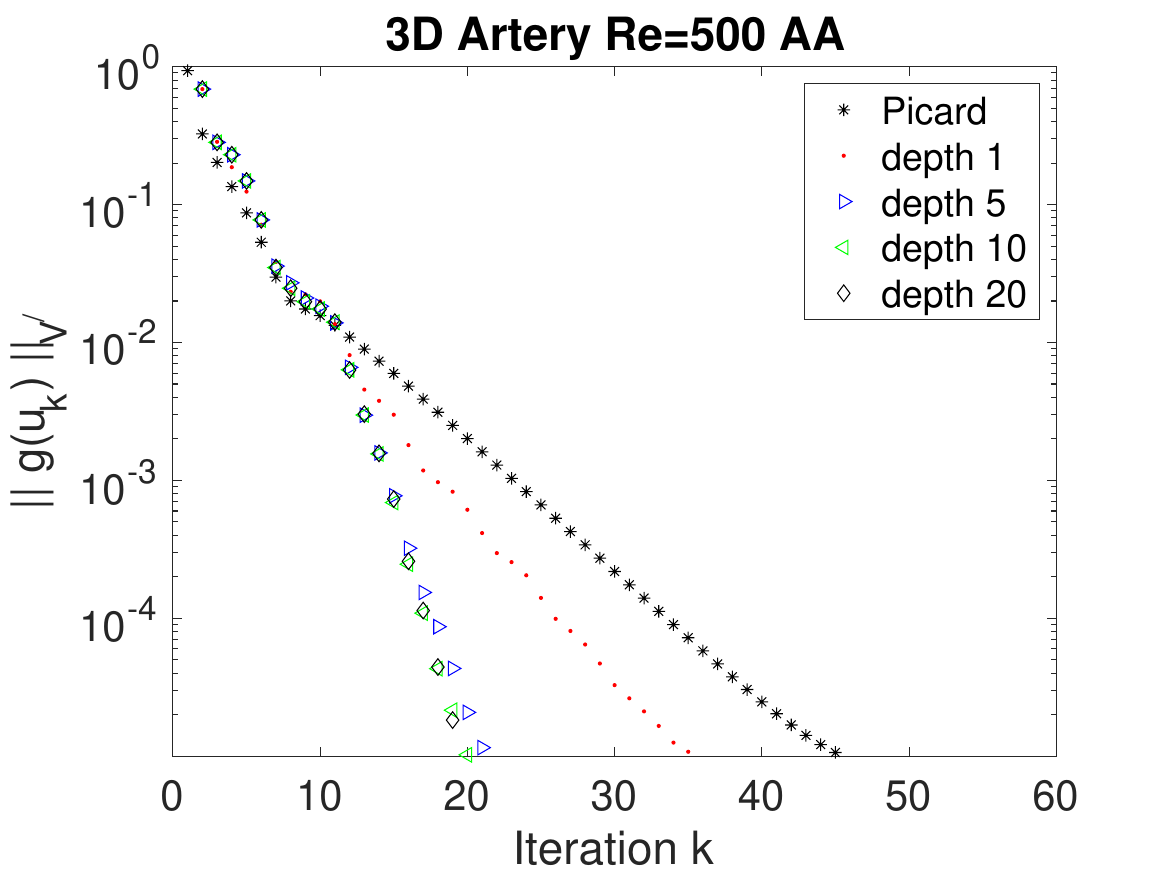}  
\includegraphics[width = .32\textwidth, height=.23\textwidth,viewport=0 0 530 420, clip]{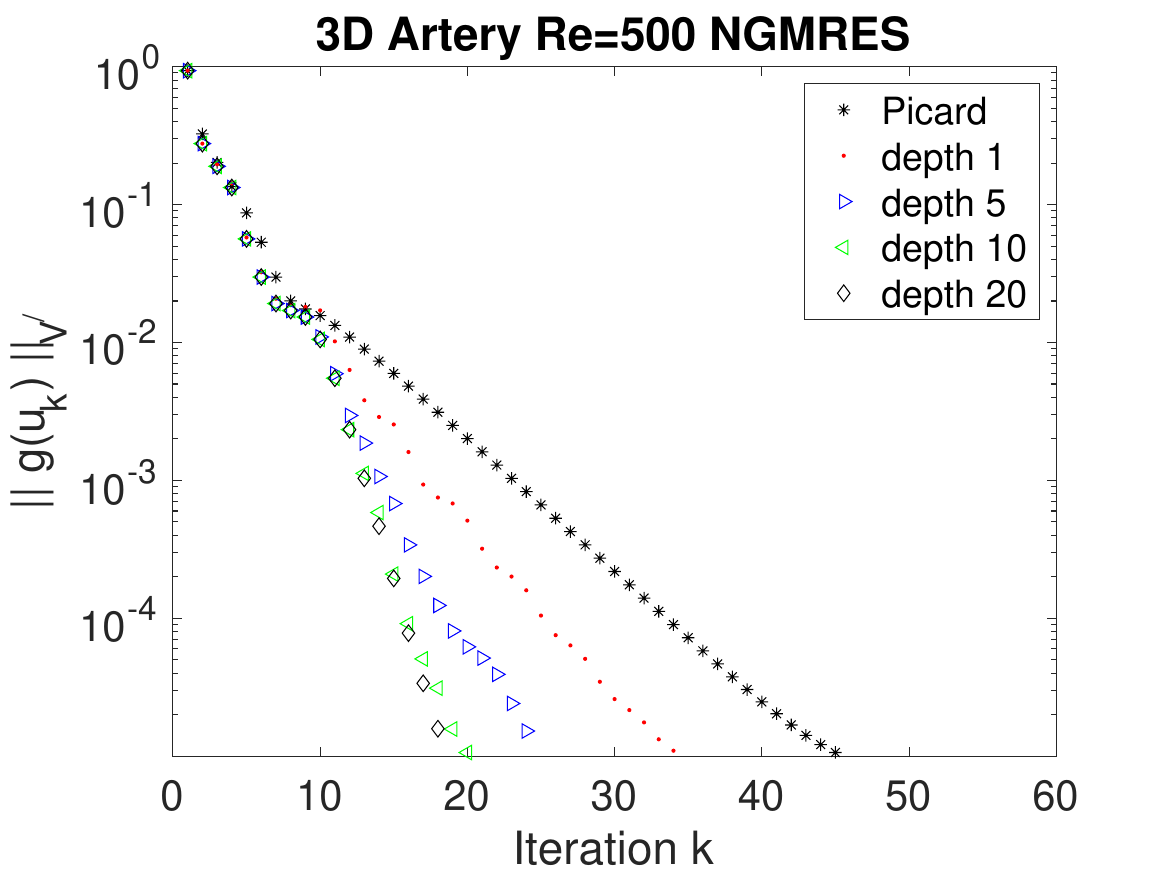} 
\caption{\label{compare1} The plots above show convergence of AAg, AA and NGMRES (left to right) applied to Picard for the NSE for 2D flow past a block, 3D driven cavity, and 3D stenotic artery.}
\end{figure}

For our next test, we compare convergence of AAg-Picard with AA-Picard and NGMRES-Picard, for each of the test problems.  Convergence is compared using the $V'$ norm, and various depths are used (noting that AA and AAg corresponds to $m=1$, but NGMRES depth 1 corresponds to $m=0$). The results are shown in Figure \ref{compare1}.
Overall, we observe similar convergence behavior between the three methods.  On 4 of the 5 tests (excluding $Re$=150 2D channel flow past a block), AA-Picard and AAg-Picard are slightly better overall compared to NGMRES-Picard.  But for $Re$=150 2D channel flow past a block, AA-Picard and AAg-Picard perform much better than NGMRES-Picard; here, AA-Picard and AAg-Picard converge in under 100 iterations (or likely will soon after 100 iterations) for depths 5 and greater, while NGMRES-Picard never converges in under 100 iterations and only for depth $\infty$ does it appear that it will eventually converge.

In addition to convergence comparison, we also compare timings.  Using AAg and NMGRES requires the optimization norm be the $V'$ norm (as shown in \cite{HR26}, bad results are obtained in 3D if $\ell^2$ is used instead).  Unfortunately, to use the $V'$ norm means that an additional Stokes solve needs performed at each iteration.  Hence a cost comparison is in order, and we did this for each of our tests.  For the 2D tests, a full LUPQ factorization (i.e. LU with row and column pivoting to minimize fill-in)
of the Stokes matrix was performed at the start and reused throughout the computation.  For 3D tests, an LUPQ factorization of the Stokes velocity block is saved, and the same type of Augmented Lagrangian solver as for the Picard solves is used for the Stokes solve.  With these approaches, the extra Stokes solve adds roughly 20\% to the total computational time for all of our tests (increase from about 6 seconds to nearly 8 seconds on average per iteration for the 2D block, and an increase from 16 to 19.5 seconds for the 3D cavity, on LR's workstation running Matlab 2023b with 196 GB RAM). 

\subsection{Adaptive depth with AAg-Picard}

\begin{figure}[h!]
\center
\includegraphics[width = .45\textwidth, height=.4\textwidth,viewport=0 0 640 520, clip]{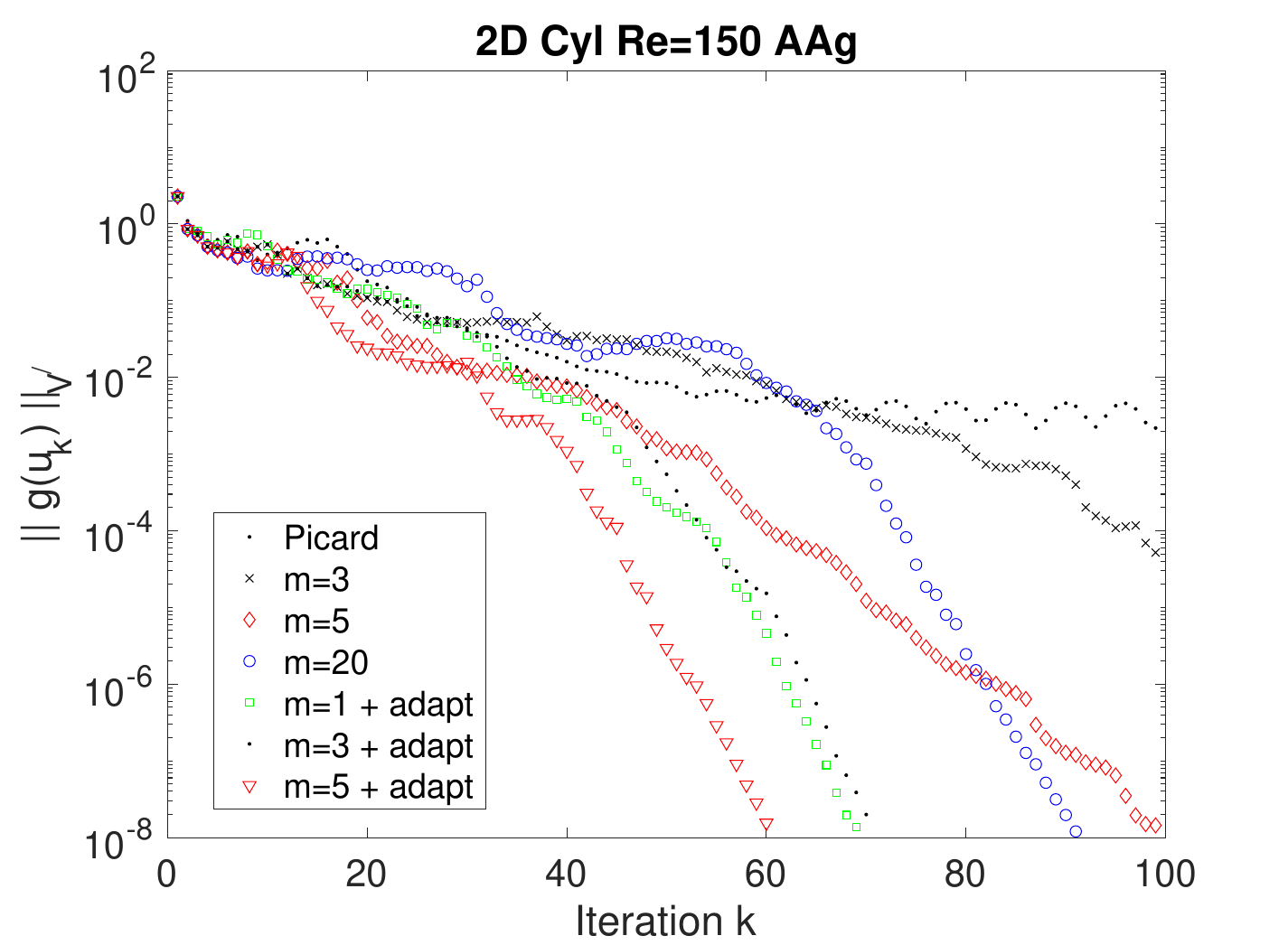}  
\includegraphics[width = .45\textwidth, height=.4\textwidth,viewport=0 0 640 520, clip]{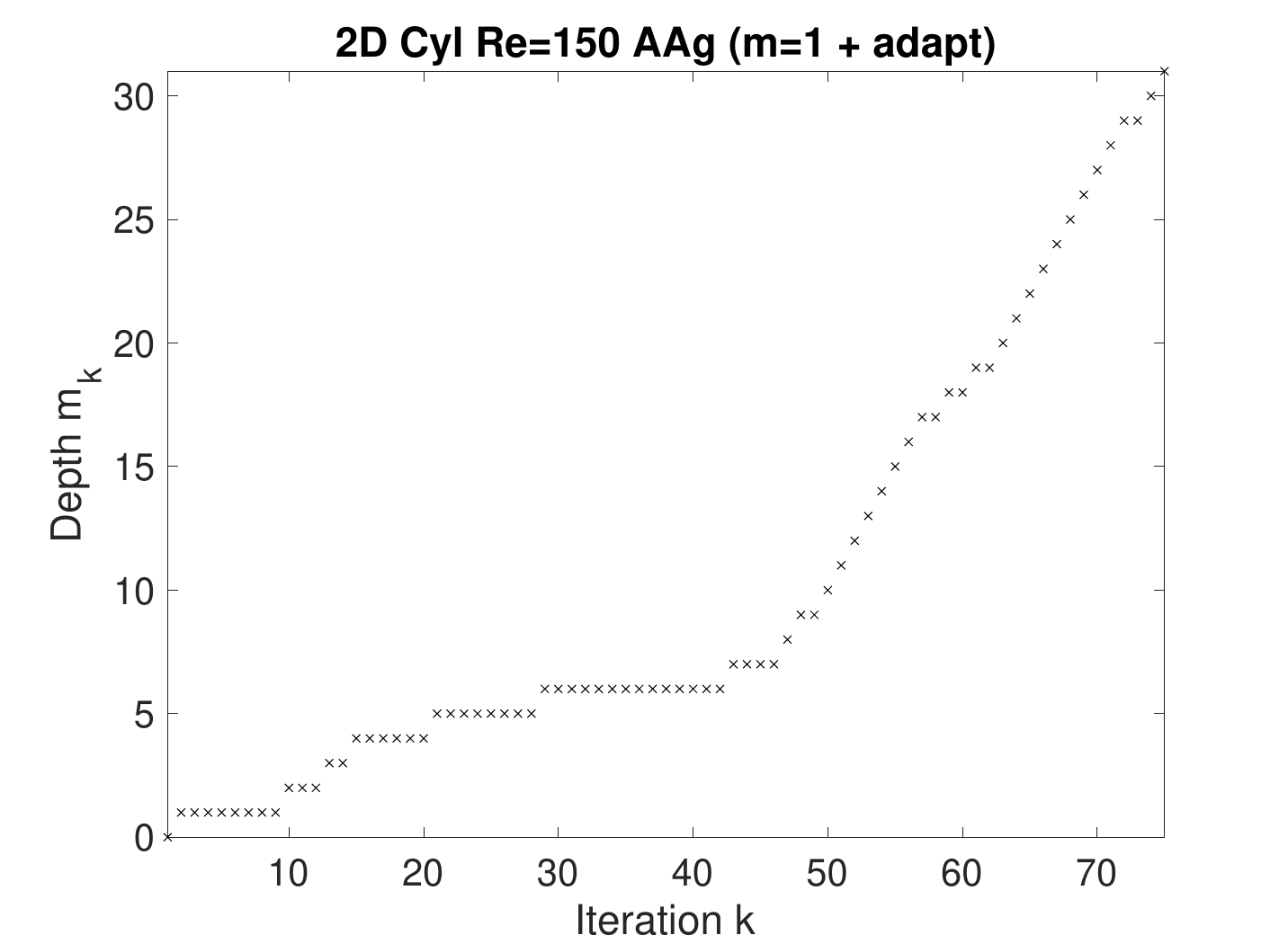}  
\caption{\label{compare2} The plots above show convergence of AAg-Picard for 2D channel flow past a block with $Re$=150 with varying constant depths and adaptive depths (left), and $m_k$ vs $k$ for the adaptive test starting with maximum $m=1$ (right).}
\end{figure}

As observed in Theorem \ref{thmm} and in the numerical tests above, $\gamma_k$ is a sharp predictor of the AAg Lipschitz constant, i.e. the linear convergence rate when the higher order terms do not play a role.  We now use this observation to construct an adaptive strategy for choosing the depth $m$ for AAg. 

Theorem \ref{thmm} shows that at iteration $k$, large $m$ can be bad when $u_{k-m+1}$ is far from the root due to the higher order terms being non-negligible and perhaps even dominant.  Hence while the iteration remains far from the root, using large $m$ can slow or even prevent convergence.  This is observed in the $Re$=150 2D channel flow past a cylinder tests, where $m=\infty$ failed but smaller $m$ performed better.  This is also a well-known phenomena for AA-Picard for NSE, where the combination of  small $m$ early and larger $m$ later in the iteration is shown to give better overall convergence \cite{PR25}.  

One advantage AAg-Picard has over AA-Picard, however, is the sharp estimate $\gamma_k$ of the Lipschitz constant.  Using this, if one checks after step $k$ that $\gamma_k \approx \frac{ \| g(u_k)\|_{V'}}{ \| g(u_{k-1})\|_{V'}}$, then the higher order terms are not playing a significant role in the convergence.  Thus, when this happens, increasing $m$ by 1 for the next iteration will help convergence by reducing the linear convergence rate contribution since the optimization is done over a larger space, but will not allow the higher order terms to become significant and have a negative impact on convergence. In the tests below we checked if 
\[
\bigg| \gamma_k - \frac{ \| g(u_k)\|_{V'}}{ \| g(u_{k-1})\|_{V'}}\bigg| <0.01,
\]
and if so then we increase $m$ by 1.

\begin{figure}[h!]
\center
\includegraphics[width = .45\textwidth, height=.4\textwidth,viewport=0 0 640 520, clip]{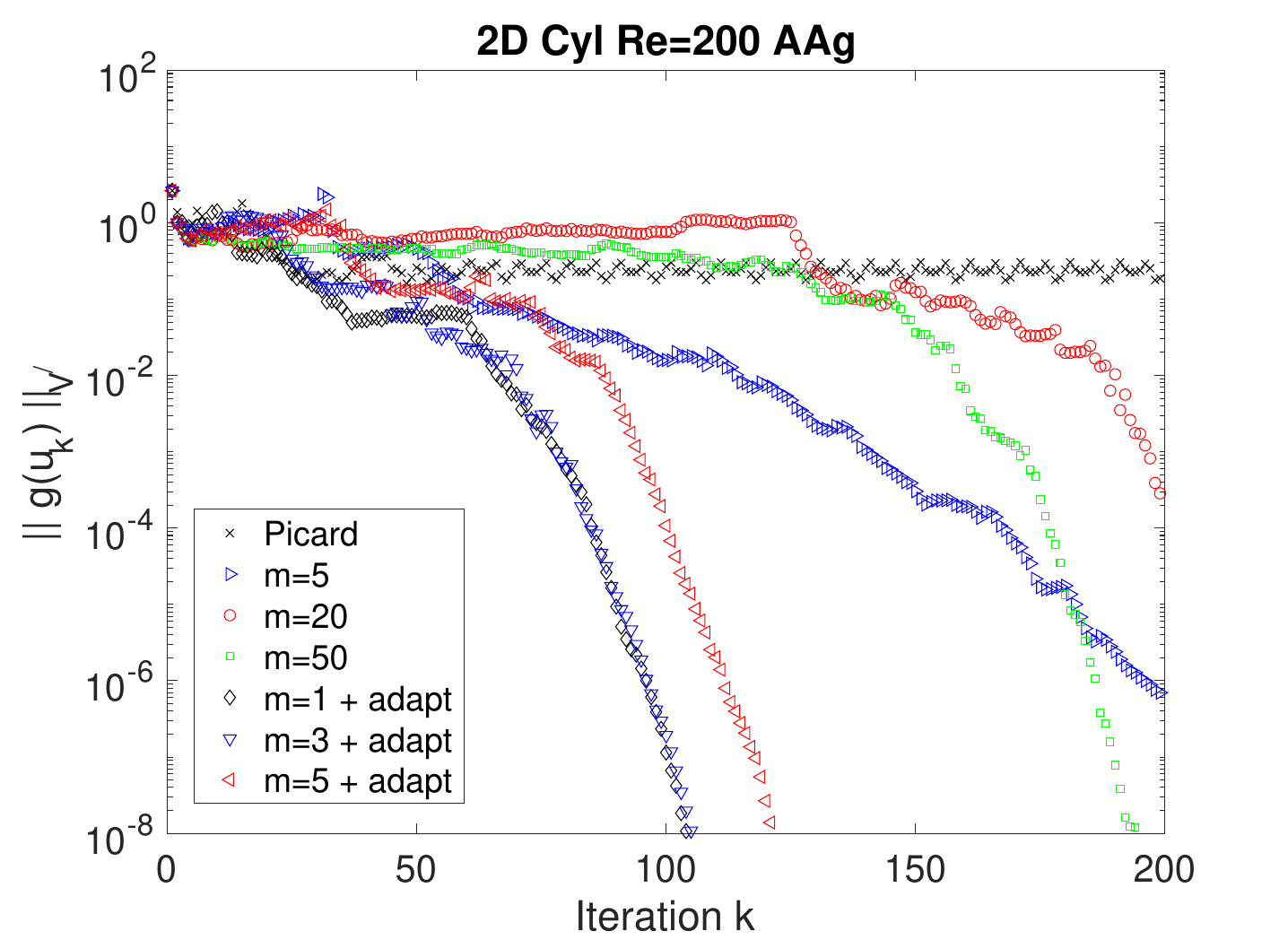}  
\includegraphics[width = .45\textwidth, height=.4\textwidth,viewport=0 0 640 520, clip]{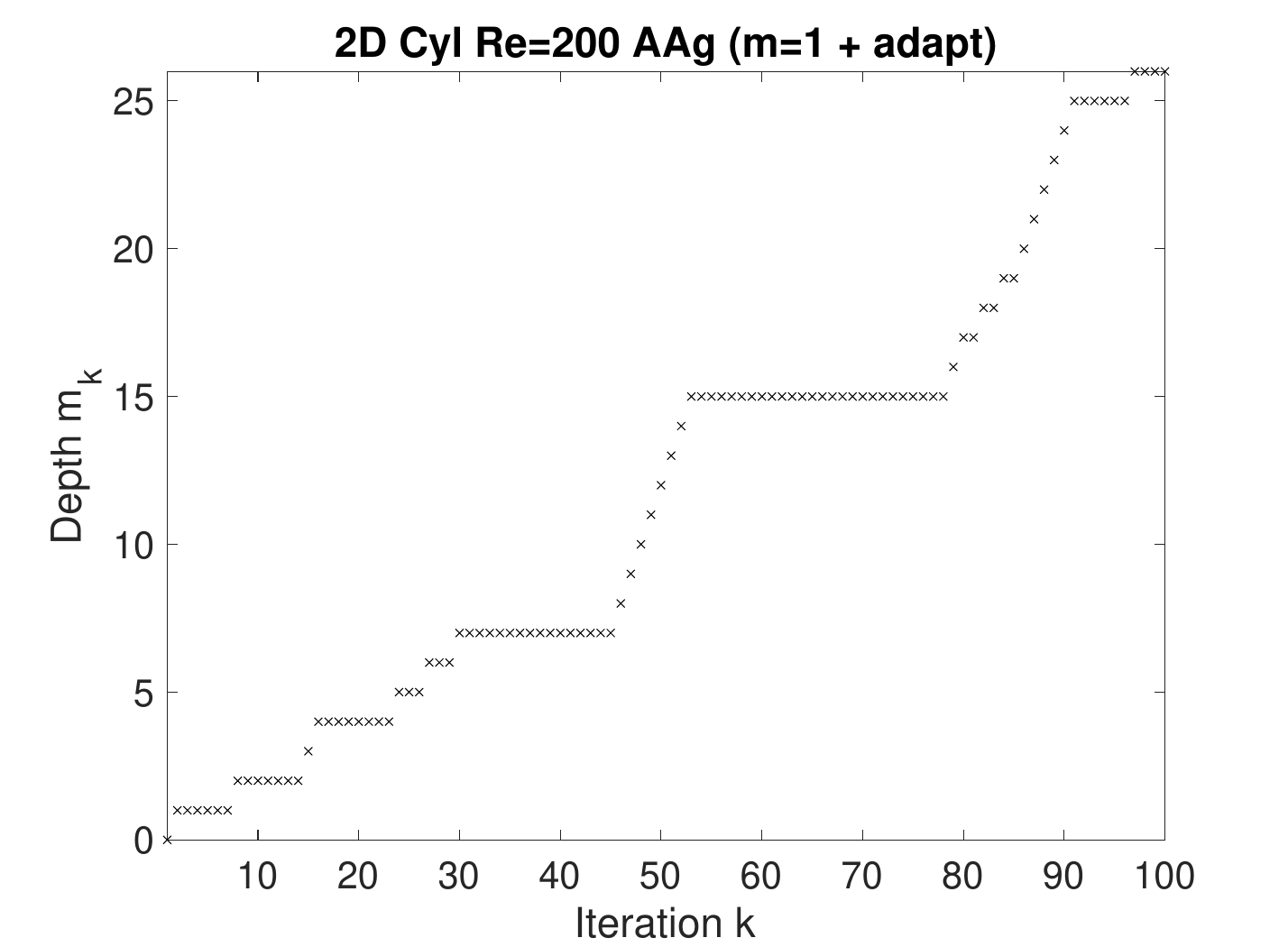}  
\caption{\label{compare3} The plots above show convergence of AAg-Picard for 2D channel flow past a block with $Re$=150 with varying constant depths and adaptive depths (left), and $m_k$ vs $k$ for the adaptive test starting with maximum $m=1$ (right).}
\end{figure}

We test this idea first on 2D channel flow past a cylinder with $Re$=150, and results are shown in Figure \ref{compare2} for AAg-Picard with varying constant depths and adaptive depths that start with a small maximum $m$=1, 3, 5.  We observe from the results that for $Re$=150, the adaptive depth computations perform much better than constant depths 3, 5, and 20 (recall from Figure \ref{compare1} that $m=20$ is the quasi-optimal choice for constant depth).  Of the three adaptive tests, using maximum $m=5$ as the starting maximum gives the best results.  Also from the figure we observe that adaptive $m$ increases slowly at first, but then by $k=50$ (when its nonlinear residual is around $10^{-3}$), it increases at almost every iteration.

The success of the adaptive strategy is even clearer for the case of $Re$=200, see Figure \ref{compare3}.  Here, AAg-Picard with constant depths $m=5$ and $m=20$ fail, with $m=50$ converging finally in 193 iterations.  The adaptive depth tests with $m=1$ and 3 as the initial maximum depth converge in 105 iterations, and the test with initial maximum of $m=5$ converged in 122.  We again observe that adaptive $m$ increases slowly at first, and then more rapidly later in the iteration. 


\begin{figure}[h!]
\center
\includegraphics[width = .45\textwidth, height=.4\textwidth,viewport=0 0 640 520, clip]{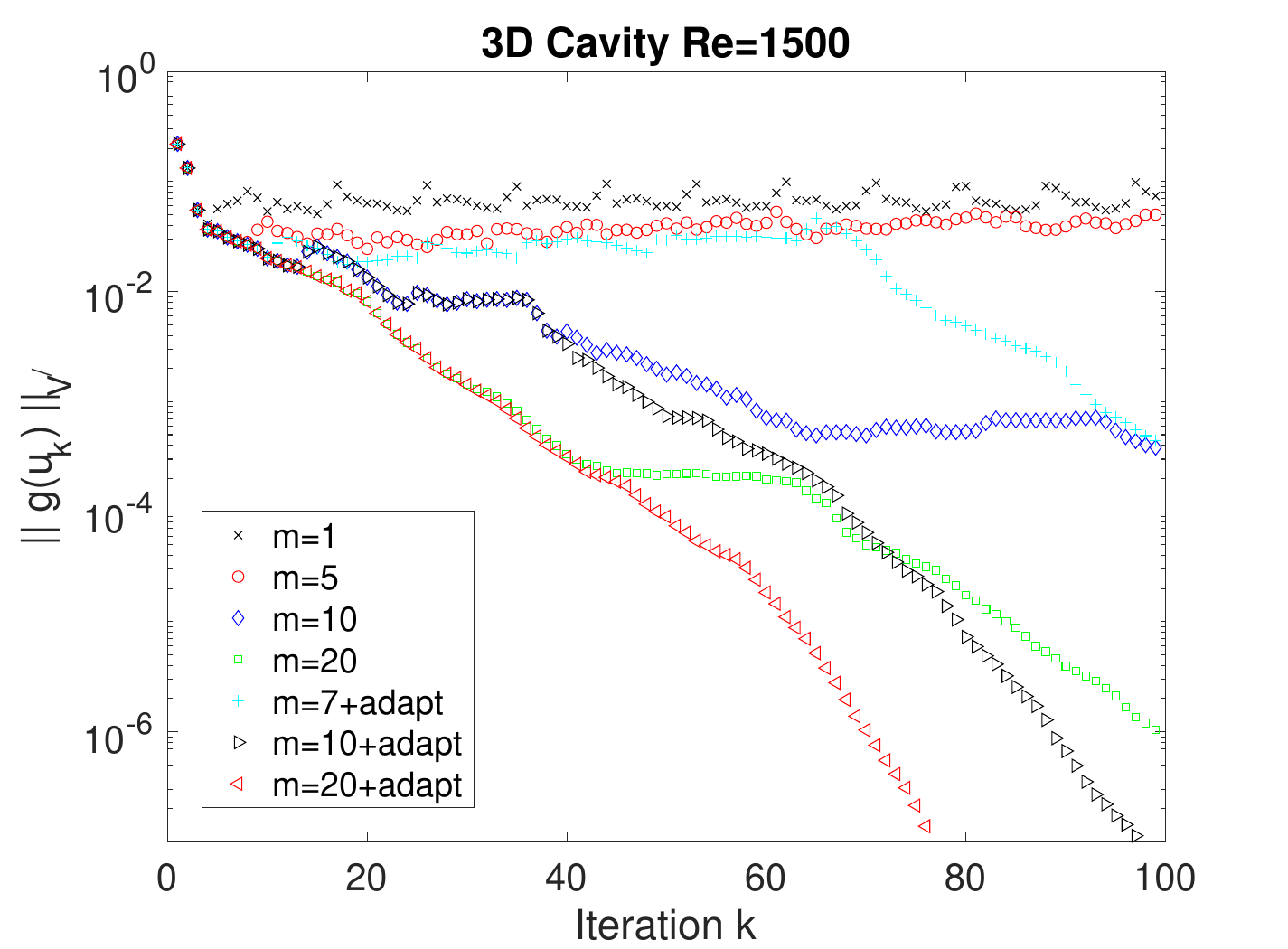}  
\includegraphics[width = .45\textwidth, height=.4\textwidth,viewport=0 0 640 520, clip]{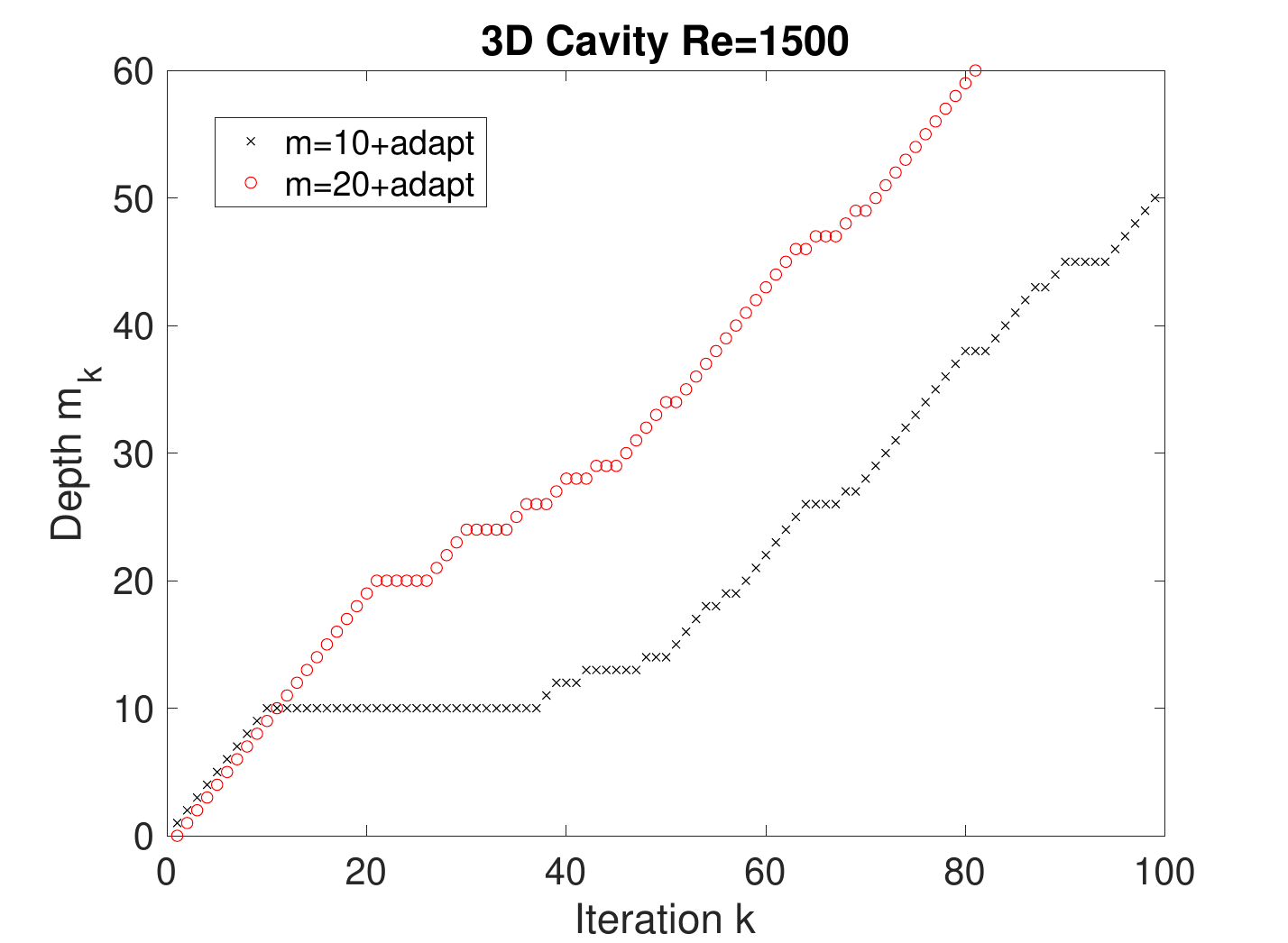}  
\caption{\label{compare4} The plots above show convergence of AAg-Picard  for 3D cavity flow with $Re$=1500, with varying constant depths and adaptive depths (left), and $m_k$ vs $k$ for the adaptive tests starting with maximum $m=10$ and 20 (right).}
\end{figure}

Our last test with the adaptive depth strategy is on the 3D cavity benchmark using $Re$=1500, and results are shown in Figure \ref{compare4}.  We observe that constant depths of $m=$1, 5 fail, $m=10$ reduces the nonlinear residual by 3 orders of magnitude and may eventually converge even though it does not converge in 100 iterations, and $m=20$ nearly converges in 100 iterations and it appears that it will converge by 120 iterations.  A key observation here is that $m$ needs to be sufficiently large to reduce the first order term coefficient in the nonlinear residual expansion, otherwise no convergence can be achieved.
For the adaptive depth tests, using $m=7$ as the initial maximum did not perform well, although finally by iteration 70 it begins to converge but still has a nonlinear residual of $10^{-3}$ after 100 iterations.  Adaptive depths using $m=10$ and 20 performed much better, especially the latter, and achieved convergence in 95 and 75 iterations, respectively.  In particular we observe that around iteration 40, the constant depth $m=20$ convergence plot and  adaptive depth with initial maximum $m=20$ convergence plot finally split off, with the adaptive test performing much better after $k=40$.

The results of the adaptive depth tests above show that using $\gamma_k$ is a cheap and effective tool to determine when $m$ can be increased so that the advantage of doing the optimization over a higher dimensional space can be realized without the penalty of higher order terms adversely affecting convergence.

\section{Conclusions}\label{sec:con}
In this work, we study a blended AA/NGMRES method called AAg analytically and numerically for the Picard iteration for the NSE, extending the initial AAg results from  \cite{he2026propertiespng}.  We provide a convergence analysis and a proof of acceleration, for general depths.  The analysis identifies the appropriate norm for the optimization problem and proves that AAg improves the linear convergence rate of the Picard iteration (and adds higher order terms).
While usual AA also has a similar theory, the AAg has a distinct feature in that the analysis provides a sharp prediction of the linear convergence at each step. The proofs of AAg convergence theory are much more similar to that of NGMRES from \cite{HR26} than AA \cite{PRX19,PR25}, i.e. they are simple and straightforward once formulated in the right way, which suggests that the differences in analysis have much more to do with using fixed point residuals vs. true nonlinear residuals (and not differences in the extrapolation definitions).

We present numerical results that illustrate  our new convergence theory, compare the performance of AAg, AA and NGMRES for several test problems, and test a new adaptive depth strategy that is based on the accuracy of the linear convergence rate prediction term. We observe similar convergence behavior between AAg, AA and NGMRES using constant depths, although NMGRES performed better using small constant depths, and for one test we observed AA and AAg performed much better than NGMRES.  
We note that due to the use of the nonlinear residuals and therefore dual norms in the optimization problem, AAg and NGMRES are mildly more expensive at each iteration due to the extra Stokes solve needed for the optimization problem.   Our tests also showed that a first attempt at an adaptive strategy based on $\gamma_k$ yielded a strong improvement in performance compared to constant depth choices. 

For future work, the new analysis of AAg and its use to develop adaptive depth strategies should be extended beyond the application of the Picard iteration for the NSE studied herein.  Hence, next steps include extending our results to create a general theory for AAg, testing it on more applications including constraint optimization, and developing more sophisticated adaptive depth strategies.  Additionally, while AAg is a blend of AA and NGMRES that uses AA extrapolation and NGMRES' nonlinear residuals in the optimization problem, the converse of NGMRES with the AA extrapolation should be explored.  In particular, this alternative AA/NGMRES blend may allow for application to superlinearly converging iterations but without the extra cost of using dual norms in the optimization problem.

\section{Declarations}

{\bf Funding:} Author LR acknowledges funding support from the US Department of Energy grant DE-SC0025292.\\ \ \\
{\bf Data Availability Statement:}  The datasets generated during and/or analyzed during the current study are available from the corresponding author on reasonable request.\\ \ \\
{\bf Code Availability Statement:}  The computer codes used during the current study are available from the corresponding author on reasonable request.


\section*{Appendix}

We now prove Lemma \ref{elem}.  This result is a generalization of the lemma from \cite[Section~3.2]{HR26}.  In that paper, the result claimed only for $u_{k-1}$ and $u_k$ in $V$ which are NGMRES iterates.  However, the proof holds exactly the same if these functions are general in  $V$, provided the $q$ is the NSE Picard iteration and $g$ is defined by \eqref{gdefn} (i.e. the exact setting we are considering herein).  For completeness, we give the proof here in the Appendix.

\begin{proof}[Proof of Lemma \ref{elem}]
Writing $d = \hat u - \hat v = \hat u - q(\hat v) + q(\hat v) - \hat v$, we have that
\[
d = (\hat u - q(\hat v)) + w,
\]
where $w=q(\hat v)-\hat v$.  Hence by the triangle inequality,
\begin{equation}
\| \nabla d \| \le \| \nabla ( \hat u - q(\hat v)) \| + \| \nabla w \|. \label{ebound}
\end{equation}
Since $w$ is a Picard residual, we have by \eqref{wg1} that
\begin{equation}
    \| \nabla w \| \le \nu^{-1} \| g(\hat v)\|_{V'}. \label{ebound1a}
\end{equation}
Thus to bound $\| \nabla d\|$, it remains only to bound $\| \nabla (\hat u - q(\hat v)) \|$.
  Using  the definitions of $g(u)$ and of \eqref{NSEPic} (with input $\hat v$ and thus solution $q(\hat v)$) gives the following two equations, for all $\chi\in V$,
\begin{align*}
\nu(\nabla \hat u,\nabla \chi) + (\hat u\cdot\nabla \hat u,\chi)  & = (f,\chi) + (g(\hat u),\chi), \\
\nu(\nabla q(\hat v),\nabla \chi) + (\hat v\cdot\nabla q(\hat v),\chi)  & = (f,\chi). 
\end{align*}
Subtracting the above equations gives  for all $\chi\in V$,
\[
\nu(\nabla (\hat u - q(\hat v)),\nabla \chi) + (\hat u\cdot\nabla (\hat u - q(\hat v)),\chi) + (d \cdot\nabla q(\hat v),\chi) = (g(\hat u),\chi).
\]
Choosing $\chi=(\hat u - q(\hat v))$ vanishes the first nonlinear term above due to skew-symmetry since $\hat v\in V$, and gives
\begin{align*}
\nu \| \nabla (\hat u - q(\hat v)) \|^2 & = - (d \cdot\nabla q(\hat v),\hat u - q(\hat v)) + (g(\hat u),\hat u - q(\hat v)) \\
& \le M \| \nabla d \| \| \nabla q(\hat v) \| \| \nabla (\hat u -q(\hat v)) \| + \| g(\hat u) \|_{V'}\| \nabla (\hat u - q(\hat v)) \|,
\end{align*}
with the last step using  \eqref{bbounds}.  After dividing by sides
by $\nu \| \nabla (\hat u - q(\hat v)) \|$, using \eqref{Picbound} to bound $\| \nabla q(\hat v) \|\le \nu^{-1} \| f \|_{-1}$, and recalling $\kappa=M\nu^{-2} \| f\|_{-1}$, this reduces to 
\[
\| \nabla (\hat u - q(\hat v)) \| \le \kappa \| \nabla d \| + \nu^{-1}  \| g(\hat u) \|_{V'}.
\]
Combining this bound with \eqref{ebound} and \eqref{ebound1a}, we get the estimate
\[
\| \nabla e \| \le \frac{\nu^{-1}}{1-\kappa} \left(  \| g(\hat u) \|_{V'} +  \| g(\hat v) \|_{V'} \right),
\]
which finishes the proof.
\end{proof}

We give here proofs of identities used in the analysis of Section 3.

\begin{proof}[Proof of Lemma \ref{ilemma}]
This proof uses the definition of $u_{k+1}$ from AAg, that $\sum_{j=k-m+1}^{k+1} \alpha_j =1$, and appropriately grouping terms.  For the first identity, expanding $u_{k+1}$ and using that $\sum_{j=k-m+1}^{k+1} \alpha_j =1$, we get that
\begin{align*}
u_{k+1} - \tilde u_{k+1} & = -(1-\alpha_{k+1}) \tilde u_{k+1} + \alpha_k \tilde u_k + ...+ \alpha_{k-m+1} \tilde u_{k-m+1}\\
& = -(1-\alpha_{k+1}) \tilde e_{k+1} - (1-\alpha_{k+1}) \tilde u_{k} + \alpha_k \tilde u_k + ...+ \alpha_{k-m+1} \tilde u_{k-m+1}\\
& = ... \\
& = -(1-\alpha_{k+1}) \tilde e_{k+1} - (1-\alpha_{k+1}-\alpha_k) \tilde e_{k} - ...- \alpha_{k-m+1} \tilde e_{k-m+2},
\end{align*}
Similarly for the second and third identities,
\begin{align*}
u_{k+1} - \tilde u_{k} & = \alpha_{k+1} \tilde u_{k+1} - (1-\alpha_k) \tilde u_k + ...+ \alpha_{k-m+1} \tilde u_{k-m+1}\\
& = \alpha_{k+1} \tilde e_{k+1} - (1-\alpha_{k+1} - \alpha_k) \tilde u_k + \alpha_{k-1} \tilde u_{k-1} ...+ \alpha_{k-m+1} \tilde u_{k-m+1}\\
& =  \alpha_{k+1} \tilde e_{k+1} - (1-\alpha_{k+1} - \alpha_k) \tilde e_k - (1-\alpha_{k+1} - \alpha_k - \alpha_{k-1}) \tilde u_{k-1} \\
& \ \ \ \ + \alpha_{k-2} \tilde u_{k-2} ...+ \alpha_{k-m+1} \tilde u_{k-m+1}\\
& = ... \\
& =  \alpha_{k+1} \tilde e_{k+1} - (1-\alpha_{k+1} - \alpha_k) \tilde e_k - (1-\alpha_{k+1} - \alpha_k - \alpha_{k-1}) \tilde e_{k-1} \\
& \ \ \ \ -  ...-  {\alpha_{k-m+1}} \tilde e_{k-m+2}
\end{align*}
\begin{align*}
u_{k+1} - \tilde u_{k-1} & = \alpha_{k+1} \tilde u_{k+1} + \alpha_k \tilde u_k - (1-\alpha_{k-1}) \tilde u_{k-1} + \alpha_{k-2} \tilde u_{k-2} + ...+ \alpha_{k-m+1} \tilde u_{k-m+1}\\
& = \alpha_{k+1} \tilde e_{k+1} +  (\alpha_{k+1} + \alpha_k) \tilde u_k - (1-\alpha_{k-1}) \tilde u_{k-1} + \alpha_{k-2} \tilde u_{k-2} + ...+ \alpha_{k-m+1} \tilde u_{k-m+1}\\
& = \alpha_{k+1} \tilde e_{k+1} +  (\alpha_{k+1} + \alpha_k) \tilde e_k - (1-\alpha_{k+1}-\alpha_{k}- \alpha_{k-1}) \tilde u_{k-1} + \alpha_{k-2} \tilde u_{k-2} \\
& \ \ \ \ + ...+ \alpha_{k-m+1} \tilde u_{k-m+1}\\
& = ... \\
& = \alpha_{k+1} \tilde e_{k+1} +  (\alpha_{k+1} + \alpha_k) \tilde e_k - (1-\alpha_{k+1}-\alpha_{k}- \alpha_{k-1}) \tilde e_{k-1} 
\\
& \ \ \ \ - ...- {\alpha_{k-m+1}} \tilde e_{k-m+2}
\end{align*}
This same process can be repeated for general $j<k$ via
\begin{align*}
u_{k+1} - \tilde u_j & =\alpha_{k+1} \tilde e_{k+1} +  (\alpha_{k+1} + \alpha_k) \tilde e_k + ... +  (\alpha_{k+1} + \alpha_k + ... +\alpha_{j+1}) \tilde e_{j+1} \\
& \ \ \ \ (1 - \alpha_{k+1} - ... - \alpha_j)\tilde e_j - ... - {\alpha_{k-m+1}}\tilde e_{k-m+2}
\end{align*}

\end{proof}

We now establish \eqref{III1}.  This result is simply an identity, although a long one, for which we use nothing more than definitions, expansions, and Lemma \ref{ilemma}.  Below, for simplicity of notation, equality is meant in the $V'$ sense, i.e. weakly when tested with an arbitary function from $V$.

Expanding $u_{k+1}$ and using that $\sum_{j=k-m+1}^{k+1}\alpha_j=1$ we obtain
\begin{align*}
g(u_{k+1}) &= u_{k+1} \cdot\nabla u_{k+1} - \nu\Delta u_{k+1} - f \\ 
&= \alpha_{k+1} \left( u_{k+1} \cdot\nabla \tilde u_{k+1} - \nu\Delta \tilde u_{k+1} - f\right) +
\alpha_k \left( u_{k+1} \cdot\nabla \tilde  u_{k} - \nu\Delta \tilde u_k - f\right)  \\
& \ \ \ \ + ... + \alpha_{k-m+1} \left(   u_{k+1}\cdot\nabla \tilde u_{k-m+1} - \nu\Delta \tilde u_{k-m+1}  - f\right) \\
& = \alpha_{k+1} g(\tilde u_{k+1}) + \alpha_{k} g(\tilde u_{k}) + ...  
+ \alpha_{k-m+1} g( \tilde u_{k-m+1} ) 
+ \alpha_{k+1}  (u_{k+1} - \tilde u_{k+1}) \cdot \nabla \tilde u_{k+1} \\
 & \ \ \ \  + \alpha_{k} (u_{k+1} - \tilde u_{k}) \cdot \nabla  \tilde u_{k} + ...
 + \alpha_{k-m+1}  (u_{k+1} - \tilde u_{k-m+1}) \cdot \nabla \tilde u_{k-m+1}. 
 \end{align*}
Hence we can write
\begin{align}
g(u_{k+1}) &= \alpha_{k+1} g(\tilde u_{k+1}) + \alpha_{k} g(\tilde u_{k}) + ...  
+ \alpha_{k-m+1} g( \tilde u_{k-m+1} )  + R, \label{Rexpansion0}
\end{align}
where
\begin{align*}
R &= \alpha_{k+1}  (u_{k+1} - \tilde u_{k+1}) \cdot \nabla \tilde u_{k+1} 
 + \alpha_{k} (u_{k+1} - \tilde u_{k}) \cdot \nabla  \tilde u_{k} + ...
 + \alpha_{k-m+1}  (u_{k+1} - \tilde u_{k-m+1}) \cdot \nabla \tilde u_{k-m+1}.
 \end{align*}
 We now analyze $R$ further.  Using Lemma \ref{ilemma} gives
\begin{align*}
 & R =   \\
 & \alpha_{k+1} \bigg( -(1-\alpha_{k+1}) \tilde e_{k+1} - (1-\alpha_{k+1}-\alpha_k) \tilde e_{k} - ...- \alpha_{k-m+1} \tilde e_{k-m+2} \bigg) \cdot \nabla \tilde u_{k+1} \\
& + \alpha_k \bigg( \alpha_{k+1} \tilde e_{k+1} - (1-\alpha_{k+1} - \alpha_k) \tilde e_k - (1-\alpha_{k+1} - \alpha_k - \alpha_{k-1}) \tilde e_{k-1} \\
& \ \ \ \  -  ...-  {\alpha_{k-m+1}} \tilde e_{k-m+2} \bigg) \cdot \nabla \tilde u_k \\
 & + \alpha_{k-1} \bigg(  \alpha_{k+1} \tilde e_{k+1} +  (\alpha_{k+1} + \alpha_k) \tilde e_k - (1-\alpha_{k+1}-\alpha_{k}- \alpha_{k-1}) \tilde e_{k-1} \\
 & \ \ \ \ - ...- {\alpha_{k-m+1}} \tilde e_{k-m+2} \bigg)  \cdot\nabla \tilde u_{k-1}\\
& + \alpha_j \bigg( \alpha_{k+1} \tilde e_{k+1} +  (\alpha_{k+1} + \alpha_k) \tilde e_k + ... +  (\alpha_{k+1} + \alpha_k + ... +\alpha_{j+1}) \tilde e_{j+1} + (1 - \alpha_{k+1} \\
& \ \ \ \ - ... - \alpha_j)\tilde e_j - ... - {\alpha_{k-m+1}}\tilde e_{k-m+2}
\bigg) \cdot\nabla \tilde u_j\\
& + ...\\
& + \alpha_{k-m+1} \bigg( \alpha_{k+1} \tilde e_{k+1} +  (\alpha_{k+1} + \alpha_k) \tilde e_k + ... +  (\alpha_{k+1} + \alpha_k + ... +\alpha_{j}) \tilde e_{j} + ... +  {\alpha_{k-m+1}}\tilde e_{k-m+2}
\bigg) \cdot\nabla \tilde u_{k-m+1}.
 \end{align*}
The terms containing a $\tilde e_{k+1}$ are as follows:
\begin{align*}
\alpha_{k+1} & \tilde e_{k+1} \cdot\nabla \bigg( -(1-\alpha_{k+1})\tilde u_{k+1} + \alpha_k \tilde u_k + \alpha_{k-1} \tilde u_{k-1} + ... + \alpha_{k-m+1} \tilde u_{k-m+1} \bigg) \\
& = \alpha_{k+1} \tilde e_{k+1} \cdot\nabla (u_{k+1} - \tilde u_{k+1}).
\end{align*}

The terms containing a $\tilde e_{k}$ are as follows:
\begin{align*}
 \tilde e_{k} & \cdot\nabla  (-\alpha_{k+1} \tilde u_{k+1} - \alpha_k \tilde u_k)
+ (\alpha_k + \alpha_{k+1})  \tilde e_k \cdot \nabla \bigg( \alpha_{k+1} \tilde u_{k+1} + ... + \alpha_{k-m+1} \tilde u_{k-m+1} \bigg) \\
& = \tilde e_{k}  \cdot\nabla (-\alpha_{k+1} \tilde u_{k+1} - \alpha_k \tilde u_k)
+ (\alpha_k + \alpha_{k+1})  \tilde e_k \cdot \nabla u_{k+1}\\
& = \alpha_k \tilde e_{k}  \cdot\nabla (u_{k+1} - \tilde u_k) + \alpha_{k+1} \tilde e_{k}  \cdot\nabla (u_{k+1} - \tilde u_{k+1}). 
\end{align*}
Similarly, the terms with $\tilde e_j$ reduce to
\[
\sum_{i=j}^{k+1} \alpha_{i}\tilde e_j \cdot\nabla (u_{k+1} - \tilde u_{i}) .
\] 
We can now write $R$ as
\begin{equation}
R = \sum_{j=k-m+2}^{k+1} \sum_{i=j}^{k+1} \alpha_{i}\tilde e_j \cdot\nabla (u_{k+1} - \tilde u_{i}).\label{Rexpansion}
\end{equation}

\end{document}